\documentclass[a4paper,11pt]{article}
\usepackage[T2A]{fontenc}
\usepackage[utf8]{inputenc}
\usepackage[russian,english]{babel}

\usepackage{latexsym}
\usepackage{geometry}
\usepackage{graphicx}
\usepackage{amsfonts}
\usepackage{amsthm}
\usepackage{amsmath}
\usepackage{amssymb}
\usepackage{xcolor}
\usepackage{mathrsfs}
\usepackage{bookmark}
\usepackage{hyperref}

\usepackage{accents}
\usepackage{indentfirst}
\usepackage{enumerate}
\usepackage{float}
\usepackage{mathtools}
\usepackage{csquotes}
\usepackage{setspace}
\usepackage{cite}
\usepackage{array}

\newtheorem{theorem}{Theorem}[section]
\newtheorem{corollary}[theorem]{Corollary}
\newtheorem{proposition}[theorem]{Proposition}
\newtheorem{lemma}[theorem]{Lemma}
\theoremstyle{definition}
\newtheorem{definition}[theorem]{Definition}
\newtheorem{descript}[theorem]{Description}

\newtheorem{notation}[theorem]{Notation}
\newtheorem{example}[theorem]{Example}
\newtheorem{remark}[theorem]{Remark}

\numberwithin{equation}{section}
\allowdisplaybreaks

\let\phi\varphi

\def\A{\mathcal{A}}

\def\Hyp{\mathcal{H}}
\def\Main{\mathcal{M}}

\renewcommand{\Im}{\mathop{\mathfrak{Im}}\nolimits}

\DeclareMathOperator{\chrs}{char}
\DeclareMathOperator{\codim}{codim}

\DeclareMathOperator{\spn}{span}
\DeclareMathOperator{\Anc}{Anc}
\DeclareMathOperator{\Ann}{Ann}

\DeclareMathOperator{\Ker}{Ker}

\newcommand\til[1]{\widetilde{\phantom{,\mkern-4mu} #1 \phantom{,\mkern-4mu}}\mkern-1mu}

\providecommand{\keywords}[1]{\textbf{Keywords:} #1}
\providecommand{\msc}[1]{\textbf{MSC 2020:} #1}

\newcommand\freefootnote[1]{%
    \bgroup
    \renewcommand\thefootnote{\fnsymbol{footnote}}%
    \renewcommand\thempfootnote{\fnsymbol{mpfootnote}}%
    \footnotetext[0]{#1}%
    \egroup
}

\begin{document}

\title{On doubly alternative zero divisors\\ in Cayley--Dickson algebras}
\author{
Svetlana Zhilina$^{a,b}$
}
\date{\small \em
$^a$Department of Mathematics and Mechanics,\\ Lomonosov Moscow State University,\\ Moscow, 119991, Russia\\
$^b$Moscow Institute of Physics and Technology,\\ Dolgoprudny, 141701, Russia
}

\maketitle

\begin{abstract}
Zero divisors of Cayley--Dickson algebras over an arbitrary field $\mathbb{F}$, $\chrs \mathbb{F} \neq 2$, are studied. It is shown that zero divisors, whose components alternate strongly pairwise and have nonzero norm, form hexagonal structures in the zero divisor graph of a Cayley--Dickson algebra. The properties of doubly alternative zero divisors, at least one of whose components has nonzero norm, are established, and an explicit form of their annihilators, orthogonalizers, and centralizers is obtained. The properties of zero divisors in Cayley--Dickson algebras with anisotropic norm are described, and it is shown that in this case directed hexagons in the zero divisor graph can be extended to undirected double hexagons in the orthogonality graph. A criterion of $C$-equivalence for elements of Cayley--Dickson algebras with anisotropic norm is obtained. Possible values of dimension for annihilators of elements of Cayley--Dickson algebras are considered.
\end{abstract}


\msc{05C25, 17A20, 17D05}

\freefootnote{This work was supported by the Theoretical Physics and Mathematics Advancement Foundation ``BASIS'' (project No. 21-8-3-8-1).}

\freefootnote{Email address: \texttt{s.a.zhilina@gmail.com}}

\section{Introduction}
A convenient method of visualization of a binary algebraic relation $R$ is to define the corresponding graph. Its vertices represent elements or their equivalence classes in an algebraic structure under consideration, and there is an edge from $x$ to $y$ if and only if $xRy$. The most popular relation graphs of various algebras are commutativity, orthogonality, and zero divisor graphs.

Studying relation graphs is a rapidly expanding branch of modern mathematics. Among the directions where relation graphs find particularly important applications, we mention the problem of classification of relation preserving mappings, cf.~\cite{dolinar_kuzma}, and the isomorphism problem, that is, exploring the connection between the isomorphism of algebraic structures and the isomorphism of corresponding relation graphs, cf.~\cite{kuzma,dolinar_kuzma1}.

This work aims to study commutativity and orthogonality relations, the relation of forming a pair of zero divisors, and the graphs induced by them, for a particular class of non-associative algebras, namely, Cayley--Dickson algebras. The study of Cayley--Dickson algebras began in the theory of composition algebras, i.e., those algebras which possess a strictly nondegenerate quadratic form~$n(\cdot)$ satisfying the identity $n(ab) = n(a) n(b)$ for all elements of the algebra.

In 1898 Hurwitz showed that the only unital composition division algebras over $\mathbb{R}$ are the real numbers~$\mathbb{R}$, the complex numbers~$\mathbb{C}$, the quaternions~$\mathbb{H}$, and the octonions~$\mathbb{O}$. Later Hurwitz theorem was extended by Jacobson to arbitrary unital composition algebras over an arbitrary field~$\mathbb{F}$, $\chrs \mathbb{F} \ne 2$. He showed that any such algebra $\A$ is isomorphic to a Cayley--Dickson algebra $\A_n$ of dimension $2^n$, where $0 \le n \le 3$, see~\cite[p.~61, Theorem~1]{Jacobson}. This result was generalized by Zhevlakov et al. to a field $\mathbb{F}$ of arbitrary characteristic, see~\cite[p.~32, Theorem~1]{Zhevlakov}.

In general, Cayley--Dickson algebras over a field $\mathbb{F}$, $\chrs \mathbb{F} \neq 2$, are a family of $2^n$-dimensional algebras $\A_n$, $n \in \mathbb{N}_0$, which are defined inductively: $\A_0 = \mathbb{F}$, and at each step the algebra $\A_{n+1}$ is obtained from $\A_n$ by applying the Cayley--Dickson process with some parameter $\gamma_n \in \mathbb{F} \setminus \{ 0 \}$. The elements of $\A_{n+1}$ are ordered pairs of elements from~$\A_n$, that is, the elements of the form $(a,b) \in \A_n \times \A_n$. For $n \geq 4$ the algebra $\A_n$ is not alternative, and thus it is not a composition algebra. Consequently, there appear zero divisors even in the case when the norm on $\A_n$ is anisotropic. The problem of their classification and description of their annihilators is not a trivial one, except for some particular cases.

At present most of the authors restrict their attention to real algebras of the main sequence which we denote by~$\Main_n$. In this case we have $\mathbb{F} = \mathbb{R}$, and all Cayley--Dickson parameters are equal to $-1$. The most successful efforts in studying zero divisors in these algebras have been taken by Moreno~\cite{moreno, moreno_alternative, moreno_constructing} and Biss, Dugger, and Isaksen~\cite{biss, biss2}. Particularly, in~\cite{biss, biss2} the dimensions of their annihilators were completely described, and the zero divisors whose annihilators have the largest possible dimension were classified. Then Pixton~\cite{pixton} obtained a similar result on the dimension of alternators in these algebras.

It should be noted that Moreno was the first to study doubly alternative elements in real algebras of the main sequence, that is, the elements whose both components are alternative in the previous algebra of the sequence. He established several important properties of doubly alternative zero divisors, see~\cite[pp. 25--27]{moreno}. One of the reasons why doubly alternative elements can be successfully studied is that, as it was shown in~\cite[p.~15]{moreno_alternative}, though the composition identity $n(ab) = n(ba) = n(a)n(b)$ does not hold in the algebra $\Main_n$ for $n \geq 4$, it is still true if $a, b \in \Main_n$ alternate with each other.

Some of the recent works on relation graphs of real Cayley--Dickson algebras include~\cite{our_split-algebras, our_sedenions, our_split-sedenions}, where relation graphs of low-dimensional real Cayley--Dickson algebras, namely, the split-complex numbers, the split-quaternions, the split-octonions, the split-sedenions, and the sedenions, have been described. In the author's papers~\cite{our_orthographs1,our_orthographs2} zero divisors in real Cayley--Dickson algebras whose components satisfy additional conditions on their norm and alternativity have been studied, and the isomorphism problem for orthogonality graphs on pairs of basis elements of real Cayley--Dickson algebras has been solved.

In the current paper we generalize the results which were obtained in~\cite{our_orthographs1} for real Cayley--Dickson algebras to the case of arbitrary Cayley--Dickson algebras over a field $\mathbb{F}$, $\chrs \mathbb{F} \neq 2$, and the results which were obtained in~\cite{moreno,biss,our_orthographs1} for real algebras of the main sequence --- to the case of arbitrary Cayley--Dickson algebras with anisotropic norm. In Corollary~\ref{corollary:alternative-subalgebras} and Lemma~\ref{lemma:A_n-alternative-properties} we correct some inaccuracies which occur in the proofs of Lemma~4.6 and Corollary~5.9 of the paper~\cite{our_orthographs1}. We also study possible values of dimension for annihilators of elements in arbitrary Cayley--Dickson algebras.

The structure of this paper is as follows: In Section~\ref{section:definitions} we introduce main definitions and notations which are used throughout the text. In particular, we describe the Cayley--Dickson process in detail in Subsection~\ref{subsection:A_n} and mention some of the properties of Cayley--Dickson algebras in Subsection~\ref{subsection:A_n-properties}.

In Section~\ref{section:alternative-subalgebras} some well-known results on subalgebras in real algebras of the main sequence are generalized to the case of arbitrary Cayley--Dickson algebras. Namely, in Lemma~\ref{lemma:quaternionic-subalgebra}, Corollary~\ref{corollary:quaternionic-subalgebra} and Theorems~\ref{theorem:nonspecial-quaternionic-subalgebra} and~\ref{theorem:octonionic-subalgebra} we determine a sufficient condition for two or three elements to generate an associative or an alternative subalgebra and present an explicit multiplication table for the elements of this subalgebra. We achieve this by constructing a homomorphism from $\A_2$ or $\A_3$ to the subalgebra discussed.

In Section~\ref{section:doubly-alternative} we consider pairs of zero divisors in arbitrary Cayley--Dickson algebras, whose components have nonzero norm and alternate with each other. It is shown in Subsection~\ref{subsection:A_n-hexagons} that they form hexagonal patterns in the zero divisor graph. Lemma~\ref{lemma:A_n-next-pair} plays a key role in studying such elements, since it allows to construct a new pair of zero divisors from a given pair of zero divisors. The main result of this subsection is Theorem~\ref{theorem:A_n-double-hexagon}. In Subsection~\ref{subsection:doubly-alternative-elements} we describe the properties of doubly alternative zero divisors, at least one of whose components has nonzero norm. Lemma~\ref{lemma:double-alternative-annihilators} and Theorem~\ref{theorem:split-algebras-commutativity-through-orthogonality} establish an explicit form of their annihilators and orthogonalizers and describe a relation between their centralizers and orthogonalizers.

In Section~\ref{section:main-sequence} we consider zero divisors in Cayley--Dickson algebras with anisotropic norm. Lemmas~\ref{lemma:moreno-zero-divisor-conditions} and~\ref{lemma:A_n-alternative-properties} generalize the results on the properties of zero divisors in real algebras of the main sequence from the paper~\cite{moreno}. Corollary~\ref{corollary:C-equivalent} shows that two noncentral elements of a Cayley--Dickson algebra with anisotropic norm are $C$-equivalent, i.e., their centralizers coincide, if and only if their imaginary parts are proportional to each other. We prove in Theorem~\ref{theorem:double-hexagon} that, in case of Cayley--Dickson algebras with anisotropic norm, directed hexagons in the zero divisor graph from Theorem~\ref{theorem:A_n-double-hexagon} can be extended to undirected double hexagons in the orthogonality graph.

Section~\ref{section:annihilator-dimensions} is devoted to studying possible values of dimension for annihilators of elements in Cayley--Dickson algebras. Examples~\ref{example:non-split-algebra-zero-divisors},~\ref{example:annihilator-dimension-2} and~\ref{example:annihilator-dimension-3} demonstrate that, in general, the dimension of an annihilator may be odd, or it may be even but not divisible by four. However, by Theorem~\ref{theorem:anisotropic-annihilator-dimension}, in case of Cayley--Dickson algebras with anisotropic norm, the dimension of an annihilator is always divisible by four. This result is a generalization of Theorem~9.8 from the paper~\cite{biss} on the dimension of annihilators in real algebras of the main sequence.

\section{Main definitions and notations} \label{section:definitions}

\subsection{Algebraic relations and their graphs} \label{subsection:definitions}

Let $\mathbb{F}$ be an arbitrary field and $(\A, +, \cdot)$ be an algebra over a field $\mathbb{F}$, possibly noncommutative and non-associative. We say that $a, b \in \A$ {\em anticommute} if $ab + ba = 0$, and $a, b \in \A$ are {\em orthogonal} if $ab = ba = 0$. We denote the set of zero divisors (left, right, or two-sided) in $\A$ by $Z(\A)$, the set of two-sided zero divisors in $\A$ by $Z_{LR}(\A)$, and the (commutative) center of $\A$ by $C_{\A}$.

\begin{definition} \label{definition:subspaces}
Let $a$ be an arbitrary element of $\A$.
\begin{itemize}
    \item
    {\em The centralizer} of $a$ is $C_\A(a) = \big\{ b \in \A \: | \: ab=ba \big\}$, i.e., the set of all elements in $\A$ which commute with $a$.
    \item
    {\em The anticentralizer} of $a$ is $\Anc_\A(a) = \big\{ b \in \A \: | \:  ab+ba=0 \big\}$, i.e., the set of all elements in $\A$ which anticommute with $a$.
    \item
    {\em The orthogonalizer} of $a$ is $O_\A(a)=\big\{ b \in \A \: | \;  ab=ba=0 \big\}$, i.e., the set of all elements in $\A$ which are orthogonal to $a$.
    \item
    {\em The left annihilator} of $a$ is the set $l.\Ann_\A(a) = \big\{ b \in \A \: | \; ba=0 \big\}$.
    \item
    Similarly, {\em the right annihilator} of $a$ is $r.\Ann_\A(a) = \big\{ b \in \A \: | \; ab=0 \big\}$.
\end{itemize}
\end{definition}

It is clear that $C_\A(a)$, $\Anc_\A(a)$, $O_\A(a)$, $l.\Ann_\A(a)$, and $r.\Ann_\A(a)$ are linear spaces over~$\mathbb{F}$.

\begin{definition}
The elements $a, b \in \A$ are called {\em $C$-equivalent} if $C_{\A}(a) = C_{\A}(b)$, and $a, b \in \A$ are called {\em $O$-equivalent} if $O_{\A}(a) = O_{\A}(b)$.
\end{definition}

\begin{notation}
For any subset $X$ of a linear space $W$ over $\mathbb{F}$ we denote the set of lines passing through elements of~$X$ by
$$
\mathbb{P}(X) = \{ [x] = \mathbb{F} x \; | \; x \in X \setminus \{ 0 \} \}.
$$
\end{notation}

We now introduce some relation graphs which are to be studied in this paper.

\begin{definition} \label{definition:graphs}
Let $\A$ be an arbitrary algebra. We define the following relation graphs of $\A$:
\begin{itemize}
\item 
{\em The commutativity graph} $\Gamma_C(\A)$: its vertices are elements of $\mathbb{P}(\A/C_\A) = \{ [a + C_{\A}] = \mathbb{F}a + C_{\A} \; | \; a \in \A \setminus C_\A \}$, and distinct vertices $[a + C_{\A}]$ and $[b + C_{\A}]$ are adjacent if and only if $ab=ba$.
\item 
{\em The orthogonality graph} $\Gamma_O(\A)$: its vertices are elements of $\mathbb{P}(Z_{LR}(\A))$, and distinct vertices $[a]$ and $[b]$ are adjacent if and only if $ab=ba=0$.
\item
{\em The directed zero divisor graph} $\Gamma_Z(\A)$: its vertices are elements of $\mathbb{P}(Z(\A))$, and distinct vertices $[a]$ and $[b]$ form a directed edge $([a], [b])$ if and only if $ab=0$.
\end{itemize}
\end{definition}

Note that the edges of $\Gamma_C(\A)$, $\Gamma_O(\A)$, and $\Gamma_Z(\A)$ are well-defined. When speaking of the vertices of these graphs, we will not distinguish between a nonzero element $a$ and a line $[a] = \mathbb{F} a$ passing through it. We also denote $\spn(a_1, \dots, a_k) = \mathbb{F} a_1 + \dots + \mathbb{F} a_k$.

\subsection{Constructing Cayley--Dickson algebras} \label{subsection:A_n}

We refer the reader to~\cite{mccrimmon,schafer} for auxiliary definitions and general properties of Cayley--Dickson algebras.

\begin{definition} \label{definition:cayley-dickson-algebras}
Let $\A$ be an algebra over a field $\mathbb{F}$ with an involution $a \mapsto \bar{a}$. The algebra $\A \{ \gamma \}$ produced by the Cayley--Dickson process, when applied to $\A$ with the parameter $\gamma \in \mathbb{F}$, $\gamma \neq 0$, is defined as the set of ordered pairs of elements of $\A$ with operations
\begin{align*}
\alpha(a,b)&=(\alpha a, \alpha b);\\
(a,b)+(c,d)&=(a+c,b+d);\\
(a,b)(c,d)&=(ac+\gamma \bar{d}b,da+b\bar{c})
\end{align*} 
and the involution
$$
\qquad (\overline{a,b})=(\bar{a},-b), \qquad a,b,c,d\in \A, \ \alpha \in \mathbb{F}.
$$
If the involution on $\A$ is regular, that is, $a + \bar{a} \in \mathbb{F}1_{\A}$ and $a\bar{a} = \bar{a}a \in \mathbb{F}1_{\A}$ for all $a \in \A$, then the involution on $\A \{ \gamma \}$ is also regular, cf.~\cite[p.~435]{schafer}.
\end{definition}

\begin{proposition} {\rm \cite[p. 161, Exercise 2.5.1]{mccrimmon} } \label{proposition:invertible-square}
Let $\gamma'=\alpha^2\gamma$ for some $\alpha \neq 0$. Then the algebras $\A \{ \gamma \}$ and $\A \{ \gamma' \}$ are isomorphic.
\end{proposition}

Henceforth we assume that $\chrs \mathbb{F} \neq 2$. We now define an arbitrary Cayley--Dickson algebra which is determined by the set of its parameters.

\begin{definition} \label{definition:A_n}
For every integer $n \geq 0$ and nonzero numbers $\gamma_0, \dots, \gamma_{n-1} \in \mathbb{F}$ we define the Cayley--Dickson algebra $\A_n = \A_n \{ \gamma_0, \dots, \gamma_{n-1} \}$ inductively:
\begin{enumerate} [(1)]
    \item $\A_0 = \mathbb{F}$, and $e^{(0)}_0 = 1$ is its only basis element;
    \item If $\A_n \{ \gamma_0, \dots, \gamma_{n-1} \}$ is constructed then $\A_{n+1} \{ \gamma_0, \dots, \gamma_n \}=(\A_n \{ \gamma_0, \dots, \gamma_{n-1} \}) \{ \gamma_n \}$. Its basis elements are $e^{(n+1)}_0, \dots, e^{(n+1)}_{2^{n+1}-1}$ such that
	\begin{equation*}
	e^{(n+1)}_m =
	\begin{cases}
	(e^{(n)}_m,0), & 0 \leq m \leq 2^n-1,\\
	(0,e^{(n)}_{m-2^n}), & 2^n \leq m \leq 2^{n+1}-1.
	\end{cases}
	\end{equation*}
\end{enumerate}
\end{definition}

For every integer $n \geq 0$ the structure $\A_n$ in Definition~\ref{definition:A_n} is a $2^n$-dimensional algebra over $\mathbb{F}$ with the unit element $e^{(n)}_0$ and a regular involution. We denote $1 = e_0 = e^{(n)}_0$ and $k = k e^{(n)}_0$ for $k \in \mathbb{F}$. 

\begin{definition} \label{definition:real-imaginary-part}
\leavevmode
\begin{itemize}
	\item Let $a \in \A_n$. Its {\em trace} is $t(a) = a + \bar{a}$, its {\em imaginary part} is $\Im(a) = \frac{a - \bar{a}}{2}$, and its {\em norm} is $n(a) = a \bar{a} = \bar{a}a$. Since the involution on $\A_n$ is regular, we have $t(a), n(a) \in \mathbb{F}$.
	\item An element $a \in \A_n$ is said to be {\em pure} if $t(a) = 0$.
	\item An element $(a, b) \in \A_{n+1}$ is said to be {\em doubly pure} if $t(a) = t(b) = 0$.
\end{itemize}
\end{definition}

\begin{proposition} {\rm \cite[p. 435]{schafer}} \label{proposition:real-cayley-dickson-properties}
We can compute trace and norm of an element $(a,b) \in \A_{n+1}$ inductively by using the following equalities:
\begin{align*}
	t((a,b)) &= t(a),\\
	n((a,b)) &= n(a) - \gamma_n n(b).
\end{align*}
\end{proposition}

It follows from Proposition~\ref{proposition:real-cayley-dickson-properties} that the norm $n(\cdot)$ is a nondegenerate quadratic form on~$\A_n$.

\subsection{Some properties of Cayley--Dickson algebras} \label{subsection:A_n-properties}

Henceforth we assume that $\A$ is an arbitrary algebra over a field $\mathbb{F}$, and $\A_n = \A_n \{ \gamma_0, \dots, \gamma_{n-1} \}$ is an arbitrary Cayley--Dickson algebra over a field $\mathbb{F}$, $\chrs \mathbb{F} \neq 2$.

\begin{proposition} {\rm \cite[p.~440]{schafer}} \label{proposition:lambda-form}
Let $\langle a, b \rangle$ denote an $\mathbb{F}$-valued symmetric bilinear form associated with the quadratic form $n(a)$. Then $\langle a, a \rangle = n(a)$ and $2\langle a,b \rangle = a \bar{b} + b \bar{a} = \bar{a} b + \bar{b} a = t(a\bar{b})$ for all $a, b \in \A_n$. Besides, for any $a, b \in \A_n$ it holds that $\langle a, b \rangle = \langle \bar{a}, \bar{b} \rangle$ and $t(a) = 2 \langle a, e_0 \rangle$.
\end{proposition}

\begin{notation}
We denote $a \perp b$ if $a$ and $b$ are orthogonal with respect to $\langle \cdot, \cdot \rangle$, that is, $\langle a, b \rangle = 0$.
\end{notation}

\begin{lemma} {\rm \cite[Lemmas~2 and~6]{schafer}} \label{lemma:A_n-zero-trace-associator} \label{lemma:inner-product-movement}
For all $x,y,z \in \A_n$ we have
\begin{enumerate}[{\rm (1)}]
    \item $t([x,y,z])=0$;
    \item $\langle x, yz \rangle = \langle x\bar{z}, y \rangle = \langle \bar{y}x, z \rangle$.
\end{enumerate}
\end{lemma}

\begin{definition} \label{definition:A_n-examples}
Let $\mathbb{F} = \mathbb{R}$.
\begin{itemize}
	\item It is said that the algebra $\A_n \{ \gamma_0, \dots, \gamma_{n-1} \}$ is {\em a real Cayley--Dickson algebra of the main sequence} if $\gamma_k = -1$ for each $k = 0, \dots, n-1$. We denote this algebra by~$\Main_n$.
	\item The algebra $\A_n \{ \gamma_0, \dots, \gamma_{n-1} \}$ is called {\em a real Cayley--Dickson split-algebra} if $\gamma_k = -1$ for each $k = 0, \dots, n-2$ and $\gamma_{n-1} = 1$.  We denote it by $\Hyp_n$, since the norm on~$\Hyp_n$ appears to be hyperbolic.
\end{itemize}
\end{definition}

\begin{proposition} {\rm \cite[Proposition~3.31]{our_split-algebras}} \label{proposition:A_n-euclidean-product}
\leavevmode
\begin{itemize}
	\item Let $a = \sum\limits_{m=0}^{2^n-1} a_{m}e^{(n)}_{m}, b = \sum\limits_{m=0}^{2^n-1} b_{m}e^{(n)}_{m} \in \Main_n$. Then $\langle a, b \rangle = \sum\limits_{m=0}^{2^n-1} a_m b_m$ is a Euclidean inner product. Particularly, $n(a) = \sum\limits_{m=0}^{2^n-1} a_{m}^2$, so $n(a)=0$ if and only if $a=0$.
	\item Let $a = \sum\limits_{m=0}^{2^n-1} a_{m}e^{(n)}_{m}, b = \sum\limits_{m=0}^{2^n-1} b_{m}e^{(n)}_{m} \in \Hyp_n$. Then $\langle a, b \rangle = \sum\limits_{m=0}^{2^{n-1}-1} a_m b_m - \sum\limits_{m=2^{n-1}}^{2^n-1} a_m b_m$.
\end{itemize}
\end{proposition}

\begin{remark}
In case of real algebras of the main sequence, the norm of $a$ is often defined as $\sqrt{a\bar{a}}$, unlike the definition $n(a)=a\bar{a}$ used in this paper. However, most of the results can be easily extended to the norm modified in this way.
\end{remark}

\begin{example}
\leavevmode
\begin{itemize}
	\item The complex numbers ($\mathbb{C}$), the quaternions ($\mathbb{H}$), the octonions ($\mathbb{O}$), and the sedenions ($\mathbb{S}$) are the real algebras of the main sequence for $n=1,\:2,\:3,$ and $4$, correspondingly, cf.~\cite{baez}.
	\item The split-complex numbers ($\hat{\mathbb{C}}$), the split-quaternions ($\hat{\mathbb{H}}$), the split-octonions ($\hat{\mathbb{O}}$), and the split-sedenions ($\hat{\mathbb{S}}$) are the real split-algebras for $n=1,\:2,\:3,$ and $4$, correspondingly, see~\cite{bentz, our_split-sedenions}.
\end{itemize}
\end{example}

We now proceed to some concepts related to associativity. For $a,b,c \in \A$ we denote their associator by $[a,b,c] = (ab)c - a(bc)$, and their anti-associator by $\{ a,b,c \} = (ab)c + a(bc)$. An algebra $\A$ is called {\em flexible} if for all $a,b \in \A$ the equality $[a,b,a] = 0$ holds. Clearly, in a flexible algebra $\A$, we have $[a,b,c]=-[c,b,a]$ for all $a,b,c \in \A$. An algebra~$\A$ is called {\em alternative} if for all $a,b \in \A$ the equalities $[a,a,b] = [b,a,a] = 0$ hold.

It is well-known that $\A_n$ is alternative if and only if $n \leq 3$, however, $\A_n$ is always flexible, see, e.g.,~\cite[p. 436, Theorem~1]{schafer}.

\begin{definition} {\rm \cite[p. 12, p. 15]{moreno_alternative}}
Let $a, b \in \A_n$.
\begin{itemize}
    \item We say that $a$ {\em alternates} with $b$ if $[a,a,b] = 0$.
    \item If $a$ alternates with every $b \in \A_n$, then $a$ is {\em alternative}.
    \item We say that $a$ {\em alternates strongly} with $b$ if $[a,a,b] = 0$ and $[b,b,a] = 0$.
    \item If $a$ alternates strongly with every $b \in \A_n$, then $a$ is {\em strongly alternative}.
\end{itemize}
\end{definition}

The following three lemmas describe the anticentralizer of an arbitrary nonzero element of $\A_n$ and the relationship between the centralizer and the orthogonalizer of an arbitrary pure element. In~\cite{our_anticomm} they are formulated for real Cayley--Dickson algebras only, however, their proofs are valid verbatim for the case of an arbitrary field. Nevertheless, we include the proofs of Lemmas~\ref{lemma:difference-centralizer-orthogonalizer} are~\ref{lemma:A_n-commutativity-through-orthogonality}, for the sake of completeness. In the formulation of Lemma~\ref{lemma:A_n-commutativity-through-orthogonality} the direct sum implies also that the direct summands are orthogonal to each other with respect to symmetric bilinear form $\langle \cdot, \cdot \rangle$. By~\cite[Proposition~8.19]{our_anticomm}, the condition $n \leq 3$ is essential in Lemma~\ref{lemma:A_n-commutativity-through-orthogonality}(1).

\begin{lemma} {\rm \cite[Lemma~5.8]{our_anticomm}} \label{lemma:A_n-anticomm}
Let $a \in \A_n$, $a \neq 0$.
\begin{enumerate}[\rm (1)]
\item If $t(a) \neq 0$, $n(a) \neq 0$, then $\Anc_{\A_n}(a) = \{ 0 \}$.
\item If $t(a) \neq 0$, $n(a)=0$, then $\Anc_{\A_n}(a) = \mathbb{F}\bar{a}$.
\item If $t(a) = 0$, then $\Anc_{\A_n}(a) = \left\{ b \in \A_n \; | \; t(b) = \langle a,b \rangle = 0 \right\} = \spn(e_0, a)^{\perp}$. \label{item:A_n-anticomm-pure}
\end{enumerate}
\end{lemma}

\begin{lemma} {\rm \cite[Lemma~8.10]{our_anticomm}} \label{lemma:difference-centralizer-orthogonalizer}
Let $x \in \A_n \setminus \{ 0 \}$, $t(x) = 0$. Then $C_{\A_n}(x) = \mathbb{F} \oplus O_{\A_n}(x) \oplus V$, where $\dim(V) \leq 1$.
\end{lemma}

\begin{proof}
It is clear that $\mathbb{F} \subseteq C_{\A_n}(x)$, so it is sufficient to show that $\Im(C_{\A_n}(x))= O_{\A_n}(x) \oplus V$, where $\dim(V) \leq 1$. By Lemma~\ref{lemma:A_n-anticomm}, $\Anc_{\A_n}(x) \subset \Im(\A_n)$, so we have
$$
O_{\A_n}(x) = C_{\A_n}(x) \cap \Anc_{\A_n}(x) = \Im(C_{\A_n}(x)) \cap \Anc_{\A_n}(x).
$$
Since for any $y \in \Im(C_{\A_n}(x))$ (and thus, $t(y)=0$) the condition $y \in \Anc_{\A_n}(x)$ is given by one linear equation, it holds that
$\dim(\Im(C_{\A_n}(x)))-\dim(O_{\A_n}(x)) \leq 1$.
\end{proof}

\begin{lemma} {\rm \cite[Lemma~8.11]{our_anticomm}} \label{lemma:A_n-commutativity-through-orthogonality}
Let $x \in \A_n \setminus \{ 0 \}$, $t(x) = 0$. Then
\begin{enumerate}[\rm (1)]
\item if $n(x) = 0$ and $n \leq 3$, then $C_{\A_n}(x) = \mathbb{F} \oplus O_{\A_n}(x)$;
\item if $n(x) \neq 0$, then $C_{\A_n}(x) = \mathbb{F} \oplus \mathbb{F}x \oplus O_{\A_n}(x)$.
\end{enumerate}
\end{lemma}

\begin{proof}
It is clear that we have an inclusion $C_{\A_n}(x) \supseteq \mathbb{F} + \mathbb{F}x + O_{\A_n}(x)$. Note that if $y \in O_{\A_n}(x)$, then $t(y) = 0$, so, by Proposition~\ref{proposition:lambda-form}, $\langle x, y \rangle = \frac{1}{2} t(x\bar{y}) = -\frac{1}{2} t(xy) = 0$. Since $n(x) = x\bar{x} = -x^2$, the conditions $n(x)=0$ and $x \in O_{\A_n}(x)$ are equivalent. Consider two cases:
\begin{enumerate}[\rm (1)]
\item If $n(x)=0$, then this inclusion takes the form $C_{\A_n}(x) \supseteq \mathbb{F} \oplus O_{\A_n}(x)$. We now show that for $n \leq 3$ the converse inclusion also holds. Let $y \in C_{\A_n}(x)$ and $t(y) = 0$. Since $n \leq 3$, we can use alternativity of $\A_n$. Note that $\overline{xy} = \bar{y}\bar{x} = yx = xy$, so $xy = k \in \mathbb{F}$. Then $0 = x^2y = x(xy) = kx$, and thus $k = 0$, that is, $y \in O_{\A_n}(x)$.
\item If $n(x) \neq 0$, then this inclusion takes the form $C_{\A_n}(x) \supseteq \mathbb{F} \oplus \mathbb{F}x \oplus O_{\A_n}(x)$. The converse inclusion follows from Lemma~\ref{lemma:difference-centralizer-orthogonalizer} for the reasons of dimension. \qedhere
\end{enumerate}
\end{proof}

\begin{example}
If $\A_n = \Main_n$ is a real algebra of the main sequence, then any element $x \in \Main_n \setminus \{ 0 \}$, $t(x) = 0$, satisfies the conditions of Lemma~\ref{lemma:A_n-commutativity-through-orthogonality}(2).
\end{example}

\section{Alternative subalgebras} \label{section:alternative-subalgebras}

In this section we determine a sufficient condition for two or three elements to generate an associative or an alternative subalgebra in an arbitrary Cayley--Dickson algebra, and we give an explicit multiplication table for the elements of this subalgebra. It should be noted that the statements~\ref{lemma:tilde-properties}, \ref{lemma:quaternionic-subalgebra}--\ref{theorem:nonspecial-quaternionic-subalgebra} and~\ref{theorem:octonionic-subalgebra} have already been partially proved in the author's paper~\cite{our_orthographs1} for the case of real Cayley--Dickson algebras. Corollary~\ref{corollary:alternative-subalgebras} has also been formulated in this paper (see~\cite[Corollary~5.9]{our_orthographs1}), but it has been assumed that the elements $x$ and $y$ alternate strongly, and its proof contained an inaccuracy: namely, there were considered orthogonal projections with respect to subspaces with possibly degenerate norm.

For $n \geq 4$ the algebra $\A_n$ is not alternative, and thus it is not a composition algebra. However, as the following lemma shows, the composition identity still holds for those elements which alternate with each other. In~\cite{moreno_alternative,our_split-algebras} it is formulated for real Cayley--Dickson algebras only, however, its proof remains valid for an arbitrary field $\mathbb{F}$, $\chrs \mathbb{F} \neq 2$.

\begin{lemma} {\rm \cite[p. 15]{moreno_alternative}, \cite[Lemma~4.8]{our_split-algebras}} \label{lemma:alternative-elements-are-normed}
Let $a, b \in \A_n$, $[a,a,b] = 0$. Then $n(ab) = n(ba) = n(a)n(b)$.
\end{lemma}

Similarly to~\cite{moreno, moreno_alternative, moreno_constructing}, we denote $\til{e}_0 = (0,e_0) \in \A_n$ and $\til{a} = a \til{e}_0$ for all $a \in \A_n$. 

\begin{lemma} \label{lemma:tilde-properties}
Let $a,b \in \A_n$, and $b$ be doubly pure. Then
\begin{enumerate}[\rm (1)]
    \item $\til{\til{a}} = \gamma_{n-1}a$;
    \item $\til{a}b = -\til{ab}$;
    \item $\til{a} \perp a$.
\end{enumerate}
If $a$ is also doubly pure, then
\begin{enumerate}[\rm (1)]
\setcounter{enumi}{3}
    \item $\til{a}b + \til{b}a = 0$ if and only if $a \perp b$;
    \item $\gamma_{n-1}ab + \til{b}\til{a} = 0$ if and only if $\til{a} \perp b$.
\end{enumerate}
\end{lemma}

\begin{proof}
Let $a=(a_1,a_2)$, $b=(b_1,b_2)$. By definition, $\til{a} = (a_1,a_2)(0,e_0) = (\gamma_{n-1}a_2,a_1)$.
\begin{enumerate}[\rm (1)]
\item It holds that $\til{\til{a}} = \til{(\gamma_{n-1}a_2,a_1)} = (\gamma_{n-1}a_1, \gamma_{n-1}a_2) = \gamma_{n-1}a$.
\item Since $b$ is doubly pure, we have
    \begin{align*}
    \til{a}b = (\gamma_{n-1}a_2,a_1)(b_1,b_2) = (\gamma_{n-1}a_2b_1 + \gamma_{n-1}\bar{b}_2a_1, \gamma_{n-1}b_2a_2 + a_1\bar{b}_1) &=\\
    = -(\gamma_{n-1}(b_2a_1 + a_2\bar{b}_1), a_1b_1 + \gamma_{n-1}\bar{b}_2a_2) &= -\til{ab}.
    \end{align*}
\item By Lemma~\ref{lemma:inner-product-movement}(2), $\langle a, \til{a} \rangle = \langle a, a\til{e}_0 \rangle = \langle \bar{a}a, \til{e}_0 \rangle = \langle n(a)e_0, \til{e}_0 \rangle = 0$.
\item By Lemma~\ref{lemma:A_n-anticomm}, $a \perp b$ if and only if $ab = -ba$, which is equivalent to $-\til{a}b = \til{ab} = -\til{ba} = \til{b}a$.
\item By Lemma~\ref{lemma:A_n-anticomm}, $\til{a} \perp b$ if and only if $\til{a}b = -b\til{a}$ or, equivalently, $-\gamma_{n-1}ab = - \til{\til{ab}} = \til{\til{a}b} = -\til{b\til{a}} = \til{b}\til{a}$. \qedhere
\end{enumerate}
\end{proof}

\begin{corollary} \label{corollary:alternative-subalgebras}
Let $x,y \in \A_{n-1}$. Then in $\A_n$ we have
$$
x\til{y} = \til{yx}, \quad \til{x}y = \til{x\bar{y}}, \quad \til{x}\til{y} = \gamma_{n-1}\bar{y}x.
$$
\end{corollary}

\begin{proof}
For any $z \in \A_{n-1}$ we have $z = (z,0)$ in $\A_n$, so, by Lemma~\ref{lemma:tilde-properties}, it holds that $\til{z} = (0, z)$. Then
\begin{align*}
    x\til{y} &= (x,0)(0,y) = (0, yx) = \til{yx},\\
    \til{x}y &= (0,x)(y,0) = (0,x\bar{y}) = \til{x\bar{y}},\\
    \til{x}\til{y} &= (0,x)(0,y) = (\gamma_{n-1}\bar{y}x,0) = \gamma_{n-1}\bar{y}x. \tag*{\qedhere}
\end{align*}
\end{proof}

Note that in Lemma~\ref{lemma:quaternionic-subalgebra} and Theorems~\ref{theorem:nonspecial-quaternionic-subalgebra} and~\ref{theorem:octonionic-subalgebra} we allow $n(a)$ and $n(b)$ to be equal to zero, in contrast to the usual definition of Cayley--Dickson algebras.

\begin{lemma} \label{lemma:quaternionic-subalgebra}
Let $a \in \A_n$ be doubly pure. Consider $\mathbb{H}_{a} = \spn(e_0,a,\til{e}_0,\til{a})$.  Then there exists a surjective homomorphism $\phi_a: \A_2 \{ -n(a), \gamma_{n-1}\} \to \mathbb{H}_{a}$, so $\mathbb{H}_{a}$ is an associative subalgebra in $\A_n$. If, moreover, $n(a) \neq 0$, then $\phi_a$ is an isomorphism.
\end{lemma}

\begin{proof}
We denote $\mu_1 = n(a)$, $\mu_2 = n(\til{e}_0) = -\gamma_{n-1}$. Since $a$ and $\til{e}_0$ are pure, we have $a^2 = -n(a) = -\mu_1$ and $(\til{e}_0)^2 = -n(\til{e}_0) = -\mu_2$. The condition $a \in \spn(e_0,\til{e}_0)^{\perp}$ implies $\til{a} \in \spn(e_0,\til{e}_0)^{\perp}$. By Lemma~\ref{lemma:tilde-properties}(3), $\til{a} \perp a$, so $a, \til{e}_0, \til{a}$ anticommute pairwise. It remains to note that $\til{a}a = -\til{aa} = \til{\mu_1 e_0} = \mu_1 \til{e}_0$ by Lemma~\ref{lemma:tilde-properties}(2), $\til{a}\til{e}_0 = \til{\til{a}} = \gamma_{n-1} a = -\mu_2 a$ by Lemma~\ref{lemma:tilde-properties}(1), and $(\til{a})^2 = -n(\til{a}) = -n(a\til{e}_0) = -n(a)n(\til{e}_0) = -\mu_1 \mu_2$ by Lemma~\ref{lemma:alternative-elements-are-normed}. Hence we have the following multiplication table in $\mathbb{H}_{a}$:
\begin{table}[H]
\centering
$
\begin{array}{|c|cccc|}
\hline
\times          & e_0             & a                     & \til{e}_0 & \til{a}          \\\hline
e_0             & e_0             & a                     & \til{e}_0 & \til{a}          \\
a               & a               & -\mu_1                & \til{a}   & -\mu_1 \til{e}_0 \\
\til{e}_0 & \til{e}_0 & -\til{a}        & -\mu_2          & \mu_2 a                \\
\til{a}   & \til{a}   & \mu_1 \til{e}_0 & -\mu_2 a        & -\mu_1 \mu_2           \\\hline
\end{array}
$
\caption{\label{table:quaternionic-subalgebra} Multiplication table in $\mathbb{H}_a$.}
\end{table}
Now we may define $\phi_a: \A_2 \{ -\mu_1, -\mu_2\} \to \mathbb{H}_a$ by $\phi_a(e_0) = e_0$, $\phi_a(e_1) = a$, $\phi_a(e_2) = \til{e}_0$, $\phi_a(e_3) = \til{a}$. The multiplication table~\ref{table:quaternionic-subalgebra} coincides with the multiplication table of $\A_2\{ -\mu_1, -\mu_2\}$, and $e_0,e_1,e_2,e_3$ form a basis in $\A_2\{ -\mu_1, -\mu_2\}$. Hence every nontrivial relation in $\A_2 \{ -\mu_1, -\mu_2\}$ is preserved under $\phi_a$, so $\phi_a$ is indeed a homomorphism. Clearly, $\phi_a$ is surjective, since $\mathbb{H}_a = \spn(e_0,a,\til{e}_0,\til{a})$.

In order to prove the last statement of the lemma, we use the fact that $e_0, a, \til{e}_0, \til{a}$ form an orthogonal system with respect to the symmetric bilinear form $\langle \cdot, \cdot \rangle$. If $n(a) \neq 0$, then $n(\til{a}) = n(a) n(\til{e}_0) \neq 0$, so $e_0, a, \til{e}_0, \til{a}$ are linearly independent, and thus $\phi_a$ is an isomorphism.
\end{proof}

\begin{remark}
Note that if $n(a) = 0$, then $\phi_a$ in Lemma~\ref{lemma:quaternionic-subalgebra} may have nontrivial kernel even for $a \neq 0$, since it is possible that $a = \til{a}$.
\end{remark}

Lemma~\ref{lemma:quaternionic-subalgebra} immediately implies a well-known statement about strong alternativity of the element $\til{e}_0$, see~\cite[Lemma~1.2]{eakin}.

\begin{corollary} \label{corollary:e_0-strongly-alternative}
The element $\til{e}_0$ is strongly alternative in $\A_n$.
\end{corollary}

\begin{proof}
Let $a \in \A_n$, and $a'$ be the orthogonal projection of $a$ onto $\spn(e_0, \til{e}_0)^{\perp}$. By Lemma~\ref{lemma:quaternionic-subalgebra}, $a'$ and $\til{e}_0$ generate an associative subalgebra $\mathbb{H}_{a'} \subset \A_n$. Clearly, $a \in \mathbb{H}_{a'}$, so $[a,a,\til{e}_0] = [\til{e}_0,\til{e}_0,a] = 0$.
\end{proof}

\begin{theorem} \label{theorem:nonspecial-quaternionic-subalgebra}
Let $a,b \in \A_n$ alternate strongly, $t(a) = t(b) = 0$. Then $\mathbb{H}_{a,b} = \spn(e_0, a, b, ab)$ is an associative subalgebra in $\A_n$ closed under involution, and its elements satisfy the multiplication table~\ref{table:nonspecial-quaternionic-subalgebra}, where $\mu_1 = n(a)$, $\mu_2 = n(b)$, and $k = - 2\langle a, b \rangle$. In the case when $k = 0$, there exists a surjective homomorphism $\psi_{a,b}: \A_2 \{ -n(a), -n(b) \} \to \mathbb{H}_{a,b}$. If, moreover, $n(a) \neq 0$ and $n(b) \neq 0$, then $\psi_{a,b}$ is an isomorphism.
\end{theorem}

\begin{proof}
Since $a$ and $b$ are pure, we have $a^2 = -n(a) = -\mu_1$ and $b^2 = -n(b) = -\mu_2$. It follows from $k = -2\langle a, b \rangle = -t(a\bar{b}) = t(ab)$ that $ba = \bar{b}\bar{a} = \overline{ab} = k - ab$.

The elements $a$ and $b$ alternate strongly, so $a(ab) = a^2b = -\mu_1 b$, $b(ab) = b(k - ba) = kb - b^2a = kb + \mu_2 a$, $(ab)a = (k-ba)a = ka - ba^2 = ka + \mu_1 b$, $(ab)b = ab^2 = -\mu_2 a$. Finally, Lemma~\ref{lemma:alternative-elements-are-normed} implies that $n(ab) = n(a)n(b) = \mu_1 \mu_2$, and thus $(ab)^2 = (ab)(k - \overline{ab}) = kab - n(ab) = kab - \mu_1 \mu_2$. Therefore, we have the following multiplication table in $\mathbb{H}_{a,b}$:
\begin{table}[H]
\centering
$
\begin{array}{|c|cccc|}
\hline
\times & e_0 & a        & b        & ab           \\\hline
e_0    & e_0 & a        & b        & ab           \\
a      & a   & -\mu_1   & ab       & -\mu_1 b     \\
b      & b   & k - ab      & -\mu_2   & kb + \mu_2 a      \\
ab     & ab  & ka + \mu_1 b  & -\mu_2 a & kab -\mu_1 \mu_2 \\\hline
\end{array}
$
\caption{\label{table:nonspecial-quaternionic-subalgebra} Multiplication table in $\mathbb{H}_{a,b}$.}
\end{table}

If $k = 0$, that is, $a \perp b$, then we may define $\psi_{a,b}: \A_2 \{ -\mu_1, -\mu_2 \} \to \mathbb{H}_{a,b}$ by $\psi_{a,b}(e_0) = e_0$, $\psi_{a,b}(e_1) = a$, $\psi_{a,b}(e_2) = b$, $\psi_{a,b}(e_3) = ab$. Then associativity of $\mathbb{H}_{a,b}$ follows from associativity of $\A_2 \{ -\mu_1, -\mu_2 \}$, and the rest of the proof is similar to that of Lemma~\ref{lemma:quaternionic-subalgebra}.

We now assume that $k \neq 0$. We first consider the case when $n(a) \neq 0$ or $n(b) \neq 0$. We may assume without loss of generality that $n(a) \neq 0$. Let $b' = b - q a$, where $q = \frac{\langle a, b \rangle}{n(a)}$. Then $a \perp b'$ and $\mathbb{H}_{a,b} = \spn(e_0, a, b, ab) = \spn(e_0, a, b', ab') = \mathbb{H}_{a,b'}$. Moreover, $[a,a,b'] = [a,a,b] - q [a,a,a] = 0$ and $[b',b',a] = [b,b,a] - q [a,b,a] - q [b,a,a] + q^2 [a,a,a] = 0$, that is, $a$ and $b'$ alternate strongly. Consequently, $\mathbb{H}_{a,b'}$ is an associative subalgebra in~$\A_n$ closed under involution, as desired.

Let now $n(a) = n(b) = 0$. Consider the elements $x = e_1 + e_2$, $y = -\frac{k}{2}(e_1 + e_3)$, $xy = \frac{k}{2}(e_0 + e_1 + e_2 + e_3)$ in $\A_2\{ -1, 1\}$. Then $e_0, x, y, xy$ are linearly independent in $\A_2\{ -1, 1\}$, and their products satisfy the same relations as the products of $e_0, a, b, ab$, since $\A_2\{ -1, 1\}$ is associative, $n(x) = n(y) = 0$ and $t(xy) = k$. Hence we may define a homomorphism $\theta_{a,b}: \A_2\{ -1, 1\} \to \mathbb{H}_{a,b}$ by $\theta_{a,b}(e_0) = e_0$, $\theta_{a,b}(x) = a$, $\theta_{a,b}(y) = b$, $\theta_{a,b}(xy) = ab$, and then associativity of $\mathbb{H}_{a,b}$ follows from associativity of $\A_2 \{ -1, 1\}$.
\end{proof}

\begin{corollary} \label{corollary:quaternionic-subalgebra}
Let $a, b \in \A_n$ alternate strongly. Then the set $\spn(e_0, a, b, ab)$ is an associative subalgebra in $\A_n$ closed under involution.
\end{corollary}

\begin{proof}
Let $a' = \Im(a)$, $b' = \Im(b)$. It is clear that $a'$ and $b'$ alternate strongly and $\spn(e_0, a, b, ab) = \spn(e_0, a', b', a'b')$. Then the desired statement follows immediately from Theorem~\ref{theorem:nonspecial-quaternionic-subalgebra}, applied to the elements $a'$ and $b'$.
\end{proof}

The next theorem is a generalization of~\cite[Theorem 5.1]{moreno_alternative}.

\begin{theorem} \label{theorem:octonionic-subalgebra}
Let $a,b \in \A_n$ be doubly pure, $b \perp \spn(a,\til{a})$. Let also $a$ alternate strongly with $b$. We denote $\mathbb{O}_{a,b} = \spn(e_0,a,b,ab,\til{e}_0,\til{a},\til{b},\til{ab})$. Then there exists a surjective homomorphism $\phi_{a,b}: \A_3 \{ -n(a), -n(b), \gamma_{n-1}\} \to \mathbb{O}_{a,b}$, so $\mathbb{O}_{a,b}$ is an alternative subalgebra in $\A_n$. If, moreover, $n(a) \neq 0$ or $n(b) \neq 0$, then $\phi_{a,b}$ is an isomorphism.
\end{theorem}

\begin{proof}
We denote $\mu_1 = n(a)$, $\mu_2 = n(b)$, $\mu_3 = n(\til{e}_0) = -\gamma_{n-1}$. Since $a, b$ and $\til{e}_0$ are pure, we have $a^2 = -n(a) = -\mu_1$, $b^2 = -n(b) = -\mu_2$, and $(\til{e}_0)^2 = -n(\til{e}_0) = -\mu_3$. We may use Lemma~\ref{lemma:inner-product-movement}(2) to show that $a \perp \spn(e_0,\til{e}_0)$ and $b \perp \spn(e_0,a,\til{e}_0,\til{a})$ imply that $\{ e_0,a,b,ab,\til{e}_0,\til{a},\til{b},\til{ab} \}$ is an orthogonal system with respect to $\langle \cdot, \cdot \rangle$. Then, by Lemma~\ref{lemma:A_n-anticomm}, $a,b,ab,\til{e}_0,\til{a},\til{b},\til{ab}$ anticommute pairwise. Note that $ab$ is also doubly pure.

By Theorem~\ref{theorem:nonspecial-quaternionic-subalgebra}, there exists a surjective homomorphism $\psi_{a,b}: \A_2 \{ -\mu_1, -\mu_2 \} \to \mathbb{H}_{a,b}$. We now extend it to $\phi_{a,b}: \A_3 \{ -\mu_1, -\mu_2, -\mu_3 \} \to \mathbb{O}_{a,b}$. We may apply Lemma~\ref{lemma:quaternionic-subalgebra} to $a,b$ and $ab$ independently. We then use Lemma~\ref{lemma:tilde-properties}(2) to obtain that $\til{a}b = -\til{ab}$, $\til{a}(ab) = -\til{a(ab)} = \mu_1 \til{b}$, $\til{b}a = -\til{ba} = \til{ab}$, $\til{b}(ab) = -\til{b(ab)} = -\mu_2 \til{a}$, $\til{ab} \cdot a = -\til{(ab)a} = -\mu_1\til{b}$, $\til{ab} \cdot b = -\til{(ab)b} = \mu_2\til{a}$. We use Lemma~\ref{lemma:tilde-properties}(5) to get that $\til{b}\til{a} = -\gamma_{n-1}ab = \mu_3 ab$, $\til{ab} \cdot \til{a} = -\gamma_{n-1}a(ab) = -\mu_1 \mu_3 b$, and $\til{ab} \cdot \til{b} = -\gamma_{n-1}b(ab) = \mu_2 \mu_3 a$. Therefore, we have the following multiplication table in $\mathbb{O}_{a,b}$:
\begin{table}[H]
\centering
$
\begin{array}{|c|cccccccc|}
\hline
\times          & e_0             & a                     & b                     & ab                          & \til{e}_0  & \til{a}          & \til{b}          & \til{ab}               \\\hline
e_0             & e_0             & a                     & b                     & ab                          & \til{e}_0  & \til{a}          & \til{b}          & \til{ab}               \\
a               & a               & -\mu_1                & ab                    & -\mu_1 b                    & \til{a}    & -\mu_1 \til{e}_0 & -\til{ab}        & \mu_1 \til{b}          \\
b               & b               & -ab                   & -\mu_2                & \mu_2 a                     & \til{b}    & \til{ab}         & -\mu_2 \til{e}_0 & -\mu_2 \til{a}         \\
ab              & ab              & \mu_1 b               & -\mu_2 a              & -\mu_1 \mu_2                & \til{ab}   & -\mu_1 \til{b}   & \mu_2 \til{a}    & -\mu_1 \mu_2 \til{e}_0 \\
\til{e}_0 & \til{e}_0 & -\til{a}        & -\til{b}        & -\til{ab}             & -\mu_3           & \mu_3 a                & \mu_3 b                & \mu_3 ab                     \\
\til{a}   & \til{a}   & \mu_1 \til{e}_0 & -\til{ab}       & \mu_1 \til{b}         & -\mu_3 a         & -\mu_1 \mu_3           & -\mu_3 ab              & \mu_1 \mu_3 b                \\
\til{b}   & \til{b}   & \til{ab}        & \mu_2 \til{e}_0 & -\mu_2 \til{a}        & -\mu_3 b         & \mu_3 ab               & -\mu_2 \mu_3           & -\mu_2 \mu_3 a               \\
\til{ab}  & \til{ab}  & -\mu_1 \til{b}  & \mu_2 \til{a}   & \mu_1 \mu_2 \til{e}_0 & -\mu_3 ab        & -\mu_1 \mu_3 b         & \mu_2 \mu_3 a          & -\mu_1 \mu_2 \mu_3           \\\hline
\end{array}
$
\caption{\label{table:octonionic-subalgebra} Multiplication table in $\mathbb{O}_{a,b}$.}
\end{table}
Hence we may define $\phi_{a,b}$ by $\phi_{a,b}((e_j,0)) = \psi_{a,b}(e_j)$ and $\phi_{a,b}((0,e_j)) = \til{\psi_{a,b}(e_j)}$ for all $0 \leq j \leq 3$. The rest of the proof is similar to that of Lemma~\ref{lemma:quaternionic-subalgebra}.
\end{proof}

\section{Zero divisors with conditions on alternativity of components} \label{section:doubly-alternative}

This section is devoted to studying zero divisors in arbitrary Cayley--Dickson algebras whose components satisfy some additional conditions on the norm and alternativity. We generalize and strengthen the results obtained in Section~3 of the author's paper~\cite{our_orthographs1} for the case of real Cayley--Dickson algebras.

\subsection{Hexagons in zero divisor graphs} \label{subsection:A_n-hexagons}

\begin{lemma} \label{lemma:A_n-next-pair}
Let $(a,b), (c,d) \in \A_{n+1}$, and the elements $c,d \in \A_n$ alternate (not strongly) with $a,b \in \A_n$. Assume also that $n(c) - \chi \gamma_n n(d) = \chi n(c) - \gamma_n n(d) = 0$ for some $\chi \in \mathbb{F}$. In this case
\begin{enumerate}[\rm (1)]
    \item if $(a,b)(c,d) = 0$, then $(c,d)(\overline{ac},-\chi da)=0$;
    \item if $(c,d)(a,b) = 0$, then $(\overline{ca},-\chi d\bar{a})(c,d)=0$.
\end{enumerate}
\end{lemma}

\begin{proof}
\leavevmode
\begin{enumerate}[\rm (1)]
\item We have the following chain of equalities
\begin{align*}
    (c,d)(\overline{ac},-\chi da) &= \left( c(\overline{ac}) + \gamma_n (\overline{-\chi da})d, (-\chi da)c + d(ac) \right) =\\
    &= \left(c(\bar{c}\bar{a}) - \chi \gamma_n (\bar{a}\bar{d})d, \chi (b\bar{c})c - \gamma_n d(\bar{d}b) \right) =\\
    &= \left((c\bar{c})\bar{a} - \chi \gamma_n \bar{a}(\bar{d}d), \chi b(\bar{c}c) - \gamma_n (d\bar{d})b \right) =\\
    &= \left((n(c) - \chi \gamma_n n(d)) \bar{a}, (\chi n(c) - \gamma_n n(d)) b \right) = 0.
\end{align*}
\item Similarly,
\begin{align*}
    (\overline{ca},-\chi d\bar{a})(c,d) &= \left( (\overline{ca})c + \gamma_n \bar{d}(-\chi d\bar{a}), d(\overline{ca}) + (-\chi d\bar{a})\bar{c} \right) =\\
    &= \left((\bar{a}\bar{c})c - \chi \gamma_n \bar{d}(d\bar{a}), d(\overline{-\gamma_n\bar{b}d}) + \chi (bc)\bar{c} \right) =\\
    &= \left(\bar{a}(\bar{c}c) - \chi \gamma_n (\bar{d}d)\bar{a}, -\gamma_n (d\bar{d})b + \chi b(c\bar{c}) \right) =\\
    &= \left((n(c) - \chi \gamma_n n(d)) \bar{a}, (\chi n(c) - \gamma_n n(d)) b \right) = 0. \tag*{\qedhere}
\end{align*}
\end{enumerate}
\end{proof}

\begin{remark} \label{remark:norm-condition}
If $n(c) = n(d) = 0$ in Lemma~\ref{lemma:A_n-next-pair}, then we can take any $\chi \in \mathbb{F}$. Otherwise, we obtain immediately 
\begin{equation*}
\begin{cases}
n(c) = \pm \gamma_n n(d) \neq 0;\\
\chi = \dfrac{n(c)}{\gamma_n n(d)} = \dfrac{\gamma_n n(d)}{n(c)} = \pm 1.
\end{cases}
\tag{\textasteriskcentered} \label{equation:norm-condition}
\end{equation*}
\end{remark}

Condition~\eqref{equation:norm-condition} is satisfied automatically if $\A_{n+1}$ is a real algebra of the main sequence, see~\cite[pp.~25--27]{moreno}, or if $\A_{n+1}$ is a real Cayley--Dickson split-algebra, see~\cite[Lemma~4.1]{our_split-algebras}. One can verify that the proofs given there require only the fact that $c,d \in \A_n$ alternate (not strongly) with $a,b \in \A_n$, and then $n(c) = n(d)$. Hence the values of $\chi$ are equal to $-1$ and $1$, respectively. It follows from Lemma~\ref{lemma:condition-asterisk} that condition~\eqref{equation:norm-condition} is also satisfied if $\A_n$ is a Cayley--Dickson algebra with anisotropic norm over an arbitrary field $\mathbb{F}$, $\chrs \mathbb{F} \neq 2$. However, condition~\eqref{equation:norm-condition} is not true in general, see~\cite[Example~4.17]{our_split-algebras} and Example~\ref{example:non-split-algebra-zero-divisors} below.

\begin{notation}
Let $(a,b) \in \A_{n+1}$ and $n(a) \neq 0$. Then the value $\chi((a,b)) = \dfrac{\gamma_n n(b)}{n(a)}$ is called the {\em characteristic} of $(a,b)$.
\end{notation}

\begin{remark}
Note that we could choose another definition of characteristic and consider the inverse of $\chi((a,b))$, i.e., $\frac{n(a)}{\gamma_n n(b)}$. Then the condition that $n(a) \neq 0$ would be replaced by $n(b) \neq 0$. Most of the results of this section can be easily transferred to the definition of characteristic modified in this way. In particular, in this case the element~$c$ would be expressed through~$d$ in Lemma~\ref{lemma:double-alternative-annihilators}.
\end{remark}

\begin{proposition} \label{proposition:self-orthogonality}
If $(x,y) \in \A_{n+1}$ is pure and $\chi((x,y)) = 1$, then $(x,y)$ is (strongly) orthogonal to itself.
\end{proposition}

\begin{proof}
By definition, $n((x,y)) = n(x) - \gamma_n n(y) = \gamma_n n(y) - \gamma_n n(y) = 0$, so $(x,y)(x,y) = -(x,y)\overline{(x,y)} = -n((x,y)) = 0$.
\end{proof}

\begin{lemma} \label{lemma:condition-asterisk}
Let the elements $c,d \in \A_n$ alternate with $a,b \in \A_n$, and let $(a,b)(c,d) = 0$ or $(c,d)(a,b) = 0$ in $\A_{n+1}$. Assume that $n(a) \neq 0$ or $n(b) \neq 0$, and also that $n(c) \neq 0$ or $n(d) \neq 0$. Then $\chi = \chi((a,b)) = \chi((c,d)) = \pm 1$ and, moreover, $\chi((\overline{ac},-\chi da)) = \chi((\overline{ca},-\chi d\bar{a})) = \chi$. In other words, the elements $(a,b)$, $(c,d)$, $(\overline{ac},-\chi da)$ and $(\overline{ca},-\chi d\bar{a})$ satisfy condition~\eqref{equation:norm-condition} with the same value of $\chi$.
\end{lemma}

\begin{proof}
We may assume without loss of generality that $(a,b)(c,d) = (ac + \gamma_n \bar{d}b, da + b\bar{c}) = 0$, since the case when $(c,d)(a,b) = 0$ is completely similar. By Lemma~\ref{lemma:alternative-elements-are-normed},
\begin{align*}
n(a)n(c) &= n(ac) = n(-\gamma_n \bar{d}b) = \gamma_n^2 n(\bar{d}b) = \gamma_n^2 n(b)n(\bar{d}) = \gamma_n^2 n(b)n(d),\\
n(a)n(d) &= n(da) = n(-b\bar{c}) = n(b\bar{c}) = n(b)n(\bar{c}) = n(b)n(c),\\
(n(c))^2n(a) &= n(c)(n(a)n(c)) = \gamma_n^2 n(c)(n(b)n(d)) = \gamma_n^2 n(d)(n(b)n(c)) =\\
&= \gamma_n^2 n(d)(n(a)n(d)) = (\gamma_n n(d))^2n(a),\\
(n(c))^2n(b) &= n(c)(n(b)n(c)) = n(c)(n(a)n(d)) = n(d)(n(a)n(c)) =\\
&= \gamma_n^2 n(d)(n(b)n(d)) = (\gamma_n n(d))^2n(b).
\end{align*}
It follows from $n(a) \neq 0$ or $n(b) \neq 0$ that $(n(c))^2 = (\gamma_n n(d))^2$, and thus $n(c) = \pm \gamma_n n(d) \neq 0$. Similarly, it follows from $n(c) \neq 0$ or $n(d) \neq 0$ that $n(a) = \pm \gamma_n n(b) \neq 0$. Hence $\chi = \chi((a,b)) = \gamma_n \dfrac{n(b)}{n(a)} = \gamma_n \dfrac{n(d)}{n(c)} = \chi((c,d)) = \pm 1$. Moreover, Lemma~\ref{lemma:alternative-elements-are-normed} implies that
$$
\chi((\overline{ac},-\chi da)) = \gamma_n \dfrac{n(-\chi da)}{n(\overline{ac})} = \gamma_n \dfrac{n(da)}{n(ac)} = \gamma_n \dfrac{n(a)n(d)}{n(a)n(c)} = \gamma_n \dfrac{n(d)}{n(c)} = \chi.
$$
It is shown similarly that $\chi((\overline{ca},-\chi d\bar{a})) = \chi$.
\end{proof}

In the statements~\ref{lemma:strongly-alternative-system}--\ref{description:A_n-double-hexagon} and in Figure~\ref{figure:directed-hexagon} we assume that $(a,b)(c,d) = 0$ in $\A_{n+1}$, and the elements $a,b \in \A_n$ alternate strongly with $c,d \in \A_n$, that is, $[x,x,y] = [y,y,x] = 0$ for $x \in \{ a, b \}$ and $y \in \{ c, d\}$. Everywhere, except for Lemma~\ref{lemma:strongly-alternative-system}, we also assume that $(a,b)$ and $(c,d)$ satisfy condition~\eqref{equation:norm-condition}.

\begin{lemma} \label{lemma:strongly-alternative-system}
The elements $ac,da$ alternate strongly with $a,b,c,d$.
\end{lemma}

\begin{proof}
It follows from $(a,b)(c,d) = (ac + \gamma_n \bar{d}b, da + b\bar{c}) = 0$ that $ac = -\gamma_n \bar{d}b$ and $da = -b\bar{c}$. It remains to apply Corollary~\ref{corollary:quaternionic-subalgebra} to the following pairs of elements: $a$ and $c$, $b$ and $c$, $a$ and $d$, $b$ and $d$.

We note that this statement can also be easily proved directly:
\begin{align*}
[a,a,ac] &= -[a,\bar{a},ac] = -(a\bar{a})(ac) + a(\bar{a}(ac)) =\\
&= -(a\bar{a})(ac) + a((\bar{a}a)c) = -n(a)ac + n(a)ac = 0,\\
[b,b,ac] &= [b,b,-\gamma_n \bar{d}b] = \gamma_n[\bar{d}b,b,b] = -\gamma_n[\bar{d}b,\bar{b},b] = 0.
\end{align*}
Similarly, all the elements $a,b,c,d$ alternate with $ac, ad$. Conversely,
\begin{align*}
    [ac,ac,a] &= -[ac, \overline{ac}, a] = -((ac)(\overline{ac}))a + (ac)((\bar{c}\bar{a})a) =\\
    &= -n(ac)a + (ac)(\bar{c}(\bar{a}a)) = -n(ac)a + n(a)(ac)\bar{c} =\\
    &= -n(ac)a + n(a)a(c\bar{c}) = -n(ac)a + n(a)n(c)a = 0,
\end{align*}
since it follows from Lemma~\ref{lemma:alternative-elements-are-normed} that $n(ac) = n(a)n(c)$. Thus it can be shown that $ac, ad$ alternate with $a,b,c,d$.
\end{proof}

\begin{corollary} \label{corollary:A_n-double-hexagon}
There exists the following $6$-cycle in $\Gamma_Z(\A_{n+1})$:
$$
(a,b) \rightarrow (c,d) \rightarrow (\overline{ac},-\chi da) \rightarrow (a,-b) \rightarrow (c,-d) \rightarrow (\overline{ac}, \chi da) \rightarrow (a,b).
$$
\end{corollary}

\begin{proof}
By Lemma~\ref{lemma:condition-asterisk}, $\chi = \chi((a,b)) = \chi((c,d)) = \chi((\overline{ac},-\chi da)) = \pm 1$. Besides, according to Lemma~\ref{lemma:strongly-alternative-system}, the elements $ac,da$ alternate strongly with $a,b,c,d$. We obtain this cycle by successive applying of Lemma~\ref{lemma:A_n-next-pair}:
\begin{itemize}
    \item $(a,b)(c,d)=0$ implies $(c,d)(\overline{ac},-\chi da)=0$;
    \item we have $\overline{c(\overline{ac})} = \overline{c(\bar{c}\bar{a})} = \overline{(c\bar{c})\bar{a}} = n(c)a$ and $-\chi (-\chi da)c = (da)c = (-b\bar{c})c = -b(\bar{c}c) = -n(c)b$, so $(c,d)(\overline{ac},-\chi da)=0$ and $n(c) \neq 0$ imply $(\overline{ac},-\chi da)(a,-b)=0$;
    \item we have $\overline{(\overline{ac})a} = \overline{(\bar{c}\bar{a})a} = \overline{\bar{c}(\bar{a}a)} = n(a)c$ and $-\chi (-b)(\overline{ac}) = \chi b (\overline{-\gamma_n\bar{d}b}) = -\chi \gamma_n b(\bar{b}d) = -\chi \gamma_n (b\bar{b})d = -\chi \gamma_n n(b)d = -n(a)d$, so $(\overline{ac},-\chi da)(a,-b)=0$ and $n(a) \neq 0$ imply $(a,-b)(c,-d)=0$;
    \item $(a,-b)(c,-d)=0$ implies $(c,-d)(\overline{ac},\chi da)=0$;
    \item $(c,-d)(\overline{ac},\chi da)=0$ implies $(\overline{ac},\chi da)(a,b)=0$. \qedhere
\end{itemize}
\end{proof}

\begin{proposition} \label{proposition:new-pairs}
Let $(x,y)(z,w) = 0$ in $\A_{n+1}$. Then $(\bar{x},\bar{y})(\gamma_n \bar{w},\bar{z}) = (\gamma_n \bar{y}, \bar{x})(\gamma_n w, z) = (\gamma_n y, x)(\bar{z}, \bar{w}) = 0$.
\end{proposition}

\begin{proof}
We have $(x,y)(z,w) = (xz + \gamma_n \bar{w}y, wx + y\bar{z}) = 0$. Hence
\begin{align*}
    (\bar{x},\bar{y})(\gamma_n \bar{w},\bar{z}) &= (\gamma_n \bar{x} \bar{w} + \gamma_n z \bar{y}, \bar{z} \bar{x} + \gamma_n \bar{y}w) = (\gamma_n (\overline{wx + y\bar{z}}), \overline{xz + \gamma_n \bar{w}y}) = 0,\\
    (\gamma_n \bar{y}, \bar{x})(\gamma_n w, z) &= (\gamma_n^2 \bar{y} w + \gamma_n \bar{z} \bar{x}, \gamma_n z \bar{y} + \gamma_n \bar{x} \bar{w}) = \gamma_n (\overline{xz + \gamma_n\bar{w} y}, \overline{wx + y\bar{z}}) = 0,\\
    (\gamma_n y, x)(\bar{z}, \bar{w}) &= (\gamma_n y \bar{z} + \gamma_n w x, \gamma_n \bar{w} y + xz) = 0. \tag*{\qedhere}
\end{align*}
\end{proof}

\begin{corollary} \label{corollary:additional-hexagons}
There exist the following $6$-cycles in $\Gamma_Z(\A_{n+1})$:
\begin{gather*}
(\bar{a},\bar{b}) \rightarrow (\gamma_n \bar{d},\bar{c}) \rightarrow (-\chi \gamma_n da,\overline{ac}) \rightarrow (\bar{a},-\bar{b}) \rightarrow (\gamma_n \bar{d},-\bar{c}) \rightarrow (\chi \gamma_n da,\overline{ac}) \rightarrow (\bar{a},\bar{b}),\\
(\gamma_n b,a) \rightarrow (\bar{c},\bar{d}) \rightarrow (-\chi \gamma_n \overline{da}, ac) \rightarrow (\gamma_n b,-a) \rightarrow (\bar{c},-\bar{d}) \rightarrow (\chi \gamma_n \overline{da}, ac) \rightarrow (\gamma_n b,a),\\
(\gamma_n \bar{b},\bar{a}) \rightarrow (\gamma_n d,c) \rightarrow (ac,-\chi \overline{da}) \rightarrow (\gamma_n \bar{b},-\bar{a}) \rightarrow (\gamma_n d,-c) \rightarrow (ac,\chi \overline{da}) \rightarrow (\gamma_n \bar{b},\bar{a}).
\end{gather*}
\end{corollary}

\begin{proof}
Follows immediately from Corollary~\ref{corollary:A_n-double-hexagon} and Proposition~\ref{proposition:new-pairs}.
\end{proof}

\begin{remark}
The cycles in Corollary~\ref{corollary:additional-hexagons} can also be obtained from Corollary~\ref{corollary:A_n-double-hexagon}, if the starting pairs of zero divisors are the pairs $(\bar{a},\bar{b})$ and $(\gamma_n \bar{d},\bar{c})$, $(\gamma_n b,a)$ and $(\bar{c},\bar{d})$, $(\gamma_n \bar{b},\bar{a})$ and $(\gamma_n d,c)$.
\end{remark}

\begin{descript} \label{description:A_n-double-hexagon}
By using Corollaries~\ref{corollary:A_n-double-hexagon} and~\ref{corollary:additional-hexagons}, we obtain the subgraphs of $\Gamma_Z(\A_{n+1})$ which we call {\em hexagons}. They are depicted in Figure~\ref{figure:directed-hexagon}.
\end{descript}

\begin{figure}[ht]
\centering
\includegraphics[width=\linewidth]{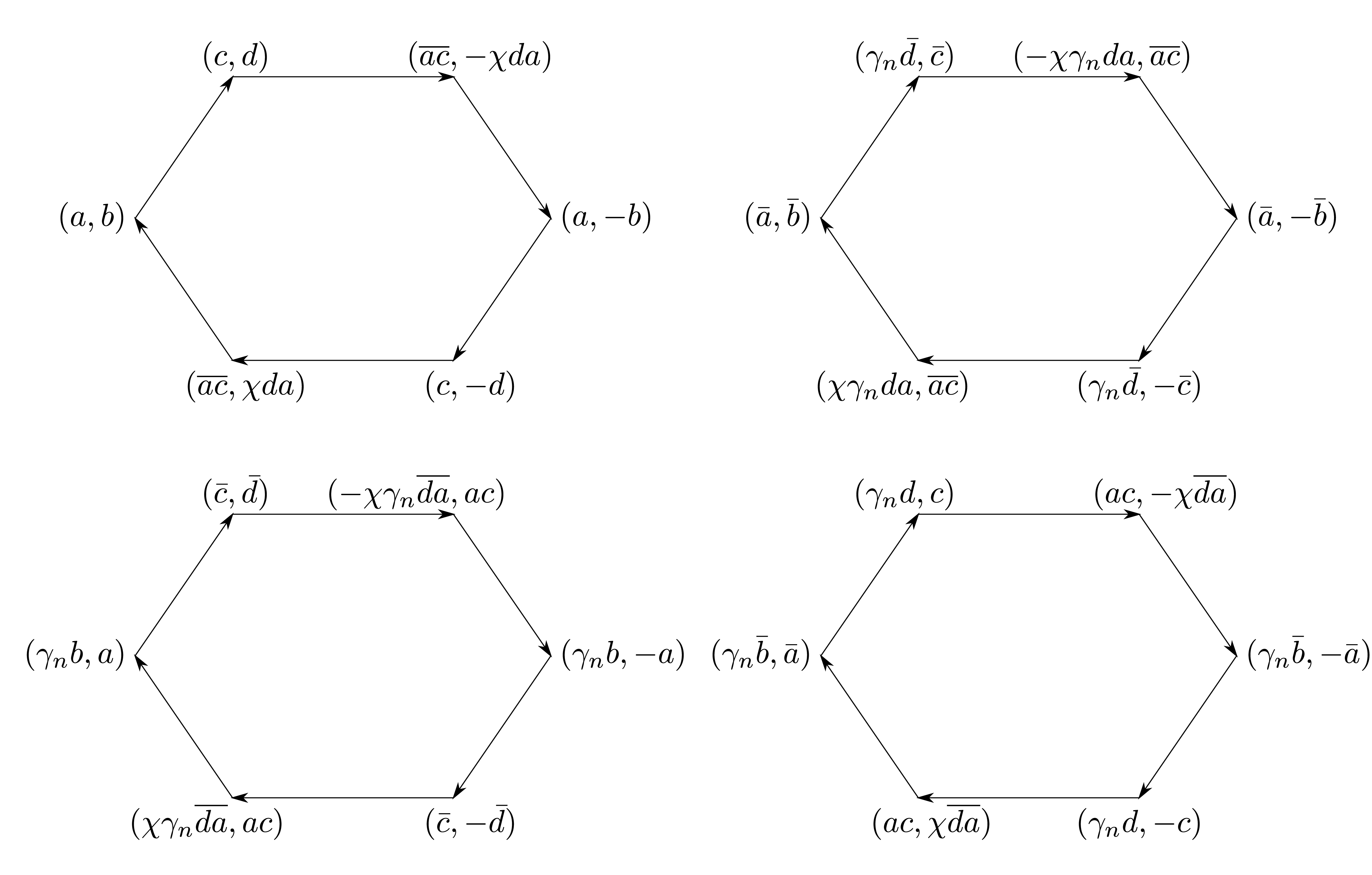}
\vspace{-0.7cm}
\caption{\label{figure:directed-hexagon} Hexagons.}
\end{figure}

Combining the results of Lemmas~\ref{lemma:condition-asterisk} and~\ref{lemma:strongly-alternative-system} and Corollaries~\ref{corollary:A_n-double-hexagon} and~\ref{corollary:additional-hexagons}, we obtain the following theorem.

\begin{theorem} \label{theorem:A_n-double-hexagon}
Let the elements $a,b \in \A_n$ alternate strongly with $c,d \in \A_n$, and $(a,b)(c,d) = 0$ in $\A_{n+1}$. Then
\begin{enumerate}[\rm (1)]
    \item The elements $ac,da$ alternate strongly with each one of the elements $a,b,c,d$.
    \item Let $n(a) \neq 0$ or $n(b) \neq 0$, and let also $n(c) \neq 0$ or $n(d) \neq 0$. Then $(a,b)$, $(c,d)$ and $(\overline{ac},-\chi da)$ satisfy condition~\eqref{equation:norm-condition} with the same value of~$\chi$.
    \item In this case there exist the following $6$-cycles in $\Gamma_Z(\A_{n+1})$:
    \begin{gather*}
    (a,b) \rightarrow (c,d) \rightarrow (\overline{ac},-\chi da) \rightarrow (a,-b) \rightarrow (c,-d) \rightarrow (\overline{ac}, \chi da) \rightarrow (a,b),\\
    (\bar{a},\bar{b}) \rightarrow (\gamma_n \bar{d},\bar{c}) \rightarrow (-\chi \gamma_n da,\overline{ac}) \rightarrow (\bar{a},-\bar{b}) \rightarrow (\gamma_n \bar{d},-\bar{c}) \rightarrow (\chi \gamma_n da,\overline{ac}) \rightarrow (\bar{a},\bar{b}),\\
    (\gamma_n b,a) \rightarrow (\bar{c},\bar{d}) \rightarrow (-\chi \gamma_n \overline{da}, ac) \rightarrow (\gamma_n b,-a) \rightarrow (\bar{c},-\bar{d}) \rightarrow (\chi \gamma_n \overline{da}, ac) \rightarrow (\gamma_n b,a),\\
    (\gamma_n \bar{b},\bar{a}) \rightarrow (\gamma_n d,c) \rightarrow (ac,-\chi \overline{da}) \rightarrow (\gamma_n \bar{b},-\bar{a}) \rightarrow (\gamma_n d,-c) \rightarrow (ac,\chi \overline{da}) \rightarrow (\gamma_n \bar{b},\bar{a}).
    \end{gather*}
\end{enumerate}
\end{theorem}

\subsection{Doubly alternative zero divisors} \label{subsection:doubly-alternative-elements}

\begin{notation}
Let $a \in \A_n$. The mappings $L_a, R_a: \A_n \rightarrow \A_n$ are given by
\begin{equation} \label{equation:multiplication-operators}
\begin{aligned}
L_a(x) &= ax,\\
R_a(x) &= xa
\end{aligned}
\end{equation}
for all $x \in \A_n$. They are linear operators on the \mbox{$2^n$-dimensional} linear space~$\A_n$.
\end{notation}

\begin{lemma} \label{lemma:two-sided-kernel}
Let $a \in \A_n$. Then $\dim(\Ker L_a) = \dim(\Ker R_a)$ or, equivalently, $\dim(l.\Ann_\A(a)) = \dim(r.\Ann_\A(a))$.
\end{lemma}

\begin{proof}
By Lemma~\ref{lemma:inner-product-movement}(2), the linear operators $L_a$ and $L_{\bar{a}}$ are conjugate with respect to the symmetric bilinear form $\langle \cdot, \cdot \rangle$ in the sense that for all $x, y \in \A_n$ we have $\langle L_a(x), y \rangle = \langle x, L_{\bar{a}}(y) \rangle$. Hence $\dim(\Ker L_a) = \dim(\Ker L_{\bar{a}})$. Besides, $L_{\bar{a}}(\bar{x}) = \bar{a}\bar{x} = \overline{xa} = \overline{R_a(x)}$ for all $x \in \A_n$, and thus $\dim(\Ker L_{\bar{a}}) = \dim(\Ker R_a)$, as required.
\end{proof}

\begin{corollary} \label{corollary:two-sided-zero-divisors}
$Z(\A_n) = Z_{LR}(\A_n)$.
\end{corollary}

\begin{proof}
Let $a \in \A_n$, $a \neq 0$. Then, by Lemma~\ref{lemma:two-sided-kernel}, $\Ker L_a \neq \{ 0 \}$ if and only if $\Ker R_a \neq \{ 0 \}$. In other words, $a$ is a right zero divisor if and only if $a$ is a left zero divisor. Hence the sets of left and right zero divisors in $\A_n$ coincide, that is, $Z(\A_n) = Z_{LR}(\A_n)$.
\end{proof}

Thus, in case of Cayley--Dickson algebras, all zero divisors appear to be two-sided zero divisors. The next proposition describes a relation between orthogonality graphs and zero divisor graphs of these algebras. Note that in~\cite{our_orthographs1} it is formulated for real Cayley--Dickson algebras only, however, its proof is valid verbatim for the case of an arbitrary field.

\begin{proposition} {\rm \cite[Proposition~3.10]{our_orthographs1}} \label{proposition:orthogonality-condition}
An edge $([a], [b])$ in $\Gamma_Z(\A_n)$ is also an edge in $\Gamma_O(\A_n)$ if and only if one of the following conditions holds:
\begin{enumerate}[(1)]
\item $[b] = [\bar{a}]$ and $n(a) = 0$;
\item $t(a) = t(b) = 0$.
\end{enumerate}
\end{proposition}

It follows from Proposition~\ref{proposition:orthogonality-condition} that any zero divisor $a \in \A_n$ with nontrivial orthogonalizer either is pure or has zero norm. If $a$ is not pure, then its connected component in $\Gamma_O(\A_n)$ consists of two vertices $[a]$ and $[\bar{a}]$. Hence, in the context of orthogonality graphs, we are interested in pure zero divisors only.

We now consider zero divisors $(a,b)\in \A_{n+1}$ such that both elements $a$ and $b$ are alternative elements in~$\A_n$.

\begin{definition} \label{definition:doubly-alternative}
The set of {\em doubly alternative elements} of $\A_{n+1}$ is
$$
DA(\A_{n+1}) = \{ (a, b) \in \A_{n+1} \; | \; \text{ both } a \text{ and }  b \; \text{are alternative in} \; \A_n \}.
$$
\end{definition}

An algebra $\A_n$ is alternative only for $n \leq 3$, so all elements of $\A_{n+1}$ are doubly alternative if and only if $n \leq 3$. Note that doubly alternative elements need not be alternative, see~\cite[Theorem~3.3]{moreno_alternative} and~\cite[Lemma~4.16]{our_split-algebras}.

Note also that, according to~\cite[Example~4.17]{our_split-algebras}, doubly alternative elements need not satisfy condition~\eqref{equation:norm-condition} even in the case when both of their components have nonzero norm. In other words, their characteristic $\chi$ can be well-defined but not equal to $0$ or $\pm 1$. However, by Lemma~\ref{lemma:condition-asterisk}, if the left or the right annihilator of some doubly alternative element $(a,b)$ contains an element $(c,d)$, $n(a) \neq 0$ or $n(b) \neq 0$, and also $n(c) \neq 0$ or $n(d) \neq 0$, then $(a,b)$ satisfies condition~\eqref{equation:norm-condition}.

\begin{lemma} \label{lemma:double-alternative-annihilators}
Let $(a,b) \in DA(\A_{n+1})$ be such that $n(a) \neq 0$. Denote $\chi = \chi((a,b))$. Then
\begin{align*}
l.\Ann_{\A_{n+1}}((a,b)) &= \left\{ \left(c, -\dfrac{(bc)a}{n(a)} \right) \; \bigg| \; b(ca) = \chi (bc)a \right\}, \\
r.\Ann_{\A_{n+1}}((a,b)) &= \left\{ \left(c, -\dfrac{(b\bar{c})\bar{a}}{n(a)} \right) \; \bigg| \; b(\bar{c}\bar{a}) = \chi (b\bar{c})\bar{a} \right\}.
\end{align*}
If, moreover, $t((a,b)) = 0$, then
$$
O_{\A_{n+1}}((a,b)) = \left\{ \left(c, -\dfrac{(bc)a}{n(a)} \right) \; \bigg| \; t(c) = 0, \: b(ca) = \chi (bc)a \right\}.
$$
\end{lemma}

\begin{proof}
We consider $l.\Ann_{\A_{n+1}}((a,b))$ first. Let $(c,d) \in \A_{n+1}$ be such that $(c,d)(a,b) = (ca + \gamma_n \bar{b}d, bc + d\bar{a}) = 0$. Then $bc + d\bar{a} = 0$, so $n(a)d = d(\bar{a}a) = (d\bar{a})a = -(bc)a$. Moreover, $ca + \gamma_n\bar{b}d = 0$ and $\chi n(a) = \gamma_n n(b)$, hence $b(ca) = -\gamma_n b(\bar{b}d) = -\gamma_n (b\bar{b})d = -\gamma_n n(b)d = -\chi n(a)d = \chi (bc)a$. Reasoning in the opposite way, we may conclude that for any $c \in \A_n$ such that $b(ca) = \chi (bc)a$ we have $\left(c, -\frac{(bc)a}{n(a)} \right) \in l.\Ann_{\A_{n+1}}((a,b))$. So, the converse is also true.

We now proceed to $r.\Ann_{\A_{n+1}}((a,b))$. Let $(c,d) \in \A_{n+1}$ be such that $(a,b)(c,d) = (ac + \gamma_n \bar{d}b, da + b\bar{c}) = 0$. Then $da + b\bar{c} = 0$, so $n(a)d = d(a\bar{a}) = (da)\bar{a} = -(b\bar{c})\bar{a}$. Since $ac + \gamma_n \bar{d}b = 0$ and $\chi n(a) = \gamma_n n(b)$, we have $b(\bar{c}\bar{a}) = b(\overline{ac}) = b(\overline{-\gamma_n \bar{d}b}) = -\gamma_n b(\bar{b}d) = -\gamma_n (b\bar{b})d = -\gamma_n n(b) d = -\chi n(a) d = \chi (b\bar{c})\bar{a}$.
Clearly, the converse is also true in this case, that is, for any $c \in \A_n$ such that $b(\bar{c}\bar{a}) = \chi (b\bar{c})\bar{a}$ we have $\left(c, -\frac{(b\bar{c})\bar{a}}{n(a)} \right) \in r.\Ann_{\A_{n+1}}((a,b))$.

Finally, the formula for $O_{\A_{n+1}}((a,b))$ follows from Proposition~\ref{proposition:orthogonality-condition}.
\end{proof}

\begin{corollary} \label{corollary:no-chords-in-hexagons}
Let $a,b,c,d,ac,ad \in \A_n$ alternate strongly pairwise, $(a,b)(c,d) = 0$ in $\A_{n+1}$, $(a,b)$ and $(c,d)$ satisfy condition~\eqref{equation:norm-condition}. Let also $t(a) = t(c) = t(ac) = 0$. Then the left upper hexagon in Figure~\ref{figure:directed-hexagon} is an undirected hexagon in $\Gamma_O(\A_{n+1})$, and there are no other edges in $\Gamma_O(\A_{n+1})$ which connect its vertices, that is, there are no chords in it.
\end{corollary}

\begin{proof}
Since $t(a) = t(c) = t(ac) = 0$, it follows from Proposition~\ref{proposition:orthogonality-condition} that this hexagon is a hexagon not only in $\Gamma_Z(\A_{n+1})$ but also in $\Gamma_O(\A_{n+1})$. 

Let now $(x,y)$ and $(z,w)$ be two of its vertices, possibly equal. We first show that $(x,y)$ cannot be orthogonal both to $(z,w)$ and to $(z,-w)$. We assume otherwise. Then, similarly to the proof of Lemma~\ref{lemma:double-alternative-annihilators}, $(x,y)(z,w) = 0$ implies $w = -\frac{(yz)x}{n(x)}$, and $(x,y)(z,-w) = 0$ implies $-w = -\frac{(yz)x}{n(x)}$, so $w = -w$. But $w \neq 0$, a contradiction.

Let $t(x) = 0$. It is possible in general that $(x,y)(x,-y) = (- n(x) - \gamma_n n(y), -2yx) = 0$, so $(x,y)$ and $(x,-y)$ are orthogonal. However, if $x$ alternates strongly with $y$ and $(x,y)$ satisfies condition~\eqref{equation:norm-condition}, then $(x,y)$ cannot be orthogonal to $(x,-y)$. Indeed, by Lemma~\ref{lemma:alternative-elements-are-normed}, $n(yx) = n(y)n(x) \neq 0$. Hence $yx \neq 0$, and thus $(x,y)(x,-y) \neq 0$.
\end{proof}

We use considerations from~\cite[p. 21]{moreno} in the proof of the following lemma.

\begin{lemma} \label{lemma:zero-self-associator}
Let $a, b, c \in \A_n$ satisfy $t(a) = t(b) = 0$, $[a,b,b] = 0$ and $b = [a,c,b]$. Then $n(b) = 0$.
\end{lemma}

\begin{proof}
Consider the mapping $S: \A_n \rightarrow \A_n$ given by $S(x) = [a,x,b]$ for all $x \in \A_n$. Then $S = R_b L_a - L_a R_b$, where $L_a$ and $R_b$ are defined by Eq.~\eqref{equation:multiplication-operators}. The linear operators $L_a$ and $R_b$ are skew-symmetric with respect to the symmetric bilinear form $\langle \cdot, \cdot \rangle$, since, by Lemma~\ref{lemma:inner-product-movement}(2), $\langle L_a(x), y \rangle = \langle ax, y \rangle = \langle x, \bar{a} y \rangle = \langle x, -a y \rangle = - \langle x, L_a y \rangle$. Therefore, $S$ is also skew-symmetric, and $S(S(x)) = 0$ implies $0 = \langle x, -S(S(x)) \rangle = \langle S(x), S(x) \rangle = n(S(x))$. Since $b = S(c)$ and $0 = [a,b,b] = S(b) = S(S(c))$, we obtain $n(b) = n(S(c)) = 0$.
\end{proof}

The next theorem is a generalization of Lemma~\ref{lemma:A_n-commutativity-through-orthogonality}(1) to the case of doubly alternative zero divisors, whose characteristic $\chi$ equals $1$. If the characteristic is well-defined but not equal to $1$, then the norm of such an element is nonzero, so Lemma~\ref{lemma:A_n-commutativity-through-orthogonality}(2) can be applied. Thus we know the explicit form of the centralizer of an arbitrary pure doubly alternative zero divisor, whose first component has nonzero norm.

\begin{theorem} \label{theorem:split-algebras-commutativity-through-orthogonality}
Let $(a,b) \in DA(\A_{n+1})$ be pure and $\chi((a,b)) = 1$. Then $C_{\A_{n+1}}((a,b)) = \mathbb{F} \oplus O_{\A_{n+1}}((a,b))$.
\end{theorem}

\begin{proof}
Since $\chi((a,b)) = 1$, we have $n((a,b)) = n(a) - \gamma_n n(b) = 0$. Assume that there exists $(c, d) \in C_{\A_{n+1}}((a,b)) \setminus (\mathbb{F} \oplus O_{\A_{n+1}}((a,b)))$. We may assume without loss of generality that $t(c) = 0$. Then
$$
\overline{(a,b) (c,d)} = \overline{(c,d)} \cdot \overline{(a,b)} = (c,d) (a,b) = (a,b) (c,d),
$$
that is, $(a,b) (c,d) = k \in \mathbb{F}$. Since $(c,d) \notin O_{\A_{n+1}}((a,b))$, we have $k \neq 0$. Assume without loss of generality that $k = 1$. Then
$$
(1,0) = (a,b) (c,d) = (ac + \gamma_n \bar{d}b, da + b\bar{c})=(ac + \gamma_n \bar{d}b, da - bc).
$$
The condition $da = bc$ implies $n(a)d = d(a\bar{a}) = (da)\bar{a} = (bc)\bar{a}$, and thus $n(a) \bar{d} = a(\bar{c}\bar{b}) = -a(c\bar{b})$. We next multiply the equality $1 = ac + \gamma_n \bar{d}b$ by $\bar{b}$ on the right and substitute the expression for $n(a) \bar{d}$:
\begin{align*}
\bar{b} &= (ac)\bar{b} + \gamma_n (\bar{d}b)\bar{b} = (ac)\bar{b} + \gamma_n \bar{d}(b\bar{b}) = (ac)\bar{b} + \gamma_n n(b) \bar{d} =\\
&= (ac)\bar{b} + n(a) \bar{d} = (ac)\bar{b} - a(c\bar{b}) = [a,c,\bar{b}].
\end{align*}
It follows from Lemma~\ref{lemma:A_n-zero-trace-associator}(1) that $t(\bar{b}) = t([a,c,\bar{b}]) = 0$, hence $\bar{b} = -b$ and $b = [a,c,b]$. By Lemma~\ref{lemma:zero-self-associator}, we have $n(b) = 0$, a contradiction with $\chi((a,b)) = 1$.
\end{proof}

\section{Zero divisors in algebras with anisotropic norm} \label{section:main-sequence}

In Lemma~\ref{lemma:moreno-zero-divisor-conditions} we extend the results of Section~1 of Moreno's paper~\cite{moreno}, which were obtained for real algebras of the main sequence, to the case of Cayley--Dickson algebras with anisotropic norm over an arbitrary field~$\mathbb{F}$, $\chrs \mathbb{F} \neq 2$. Recall that we denote $\til{e}_0 = (0,e_0) \in \A_n$ and $\til{a} = a \til{e}_0$ for all $a \in \A_n$.

Corollary~\ref{corollary:two-sided-zero-divisors} shows that in Cayley--Dickson algebras all zero divisors appear to be two-sided zero divisors, that is, $Z(\A_n) = Z_{LR}(\A_n)$. In case of Cayley--Dickson algebras with anisotropic norm a stronger result holds.

\begin{lemma} {\rm \cite[Corollaries~1.5,~1.6,~1.9 and~1.12]{moreno}}  \label{lemma:moreno-zero-divisor-conditions} \label{lemma:moreno-doubly-pure} \label{lemma:moreno-tilde}
Let $\A_n$ be a Cayley--Dickson algebra with anisotropic norm, $a, b \in \A_n$. Then
\begin{enumerate}[{\rm (1)}]
    \item $n(ab) = n(\bar{a}b) = n(a\bar{b}) = n(ba)$;
    \item the elements $ab$, $ba$, $\bar{a}b$, $a\bar{b}$ are equal to zero or not equal to zero simultaneously;
    \item if $a \in Z(\A_n)$, then $t(a) = 0$;
    \item if $a \in Z(\A_n)$, then $a$ is doubly pure;
    \item $ab = 0$ if and only if $a\til{b} = 0$.
\end{enumerate}
\end{lemma}

\begin{proof}
\leavevmode
\begin{enumerate}[{\rm (1)}]
    \item By Lemma~\ref{lemma:inner-product-movement}(2), $n(ab) = \langle ab, ab \rangle = \langle \bar{a}(ab), b \rangle = \langle \bar{b}(\bar{a}(ab)), e_0 \rangle = \frac{1}{2} t(\bar{b}(\bar{a}(ab)))$. Note that $\bar{a}(ab) = (t(a) - a)(ab) = a((t(a) - a)b) = a(\bar{a}b)$, hence $n(ab) = \frac{1}{2} t(\bar{b}(\bar{a}(ab))) = \frac{1}{2} t(\bar{b}(a(\bar{a}b))) = n(\bar{a}b)$. Moreover, $n(\bar{a}b) = n(\overline{\bar{a}b}) = n(\bar{b}a) = n(ba) = n(\overline{ba}) = n(\bar{a}\bar{b}) = n(a\bar{b})$.
    \item Follows immediately from item~(1) and the fact that the norm on $\A_n$ is anisotropic.
    \item Let $a \in Z(\A_n)$, that is, $ab = 0$ for some $b \in \A_n$, $b \neq 0$. According to item~(2), $ab = \bar{a}b = 0$, so $t(a)b = (a + \bar{a})b = 0$. Since $t(a) \in \mathbb{F}$ and $b \neq 0$, we obtain $t(a) = 0$.
    \item It follows from Proposition~\ref{proposition:new-pairs} that if $a = (a_1,a_2) \in Z(\A_n)$, then $\til{a} = (\gamma_{n-1} a_2, a_1) \in Z(\A_n)$. Therefore, by item~(3), we have $t(a) = t(\til{a}) = 0$, that is, $t(a_1) = t(a_2) = 0$.
    \item If $a = 0$ or $b = 0$, then $ab = a\til{b} = 0$. Otherwise, by item~(4), $a$ and~$b$ are doubly pure, so the required statement follows from the first equality in Proposition~\ref{proposition:new-pairs}. \qedhere
\end{enumerate}
\end{proof}

\begin{corollary} \label{corollary:zero-divisors-symmetry}
Let $\A_n$ be a Cayley--Dickson algebra with anisotropic norm. Then $\Gamma_Z(\A_n)$ can be obtained from $\Gamma_O(\A_n)$ by replacing every undirected edge with a pair of directed edges.
\end{corollary}

\begin{proof}
Follows immediately from Lemma~\ref{lemma:moreno-zero-divisor-conditions}(2).
\end{proof}

\begin{theorem} \label{theorem:C-equivalent}
Let $\A_n$ be a Cayley--Dickson algebra with anisotropic norm, $a, b \in \A_n \setminus \{ 0 \}$, $t(a) = t(b) = 0$. If $a$ and $b$ are $C$-equivalent, i.e, $C_{\A_n}(a) = C_{\A_n}(b)$, then $[a] = [b]$.
\end{theorem}

\begin{proof}
By Lemma~\ref{lemma:A_n-commutativity-through-orthogonality}(2), $C_{\A_n}(a) = \mathbb{F} \oplus \mathbb{F}a \oplus O_{\A_n}(a)$. Moreover, by Proposition~\ref{proposition:orthogonality-condition}, for all $x \in O_{\A_n}(a)$ it holds that $t(x) = 0$. Since $b \in C_{\A_n}(b) = C_{\A_n}(a)$ and $t(b) = 0$, we have $b = ka + x$ for some $k \in \mathbb{F}$ and $x \in O_{\A_n}(a)$. Therefore, $C_{\A_n}(a) \subseteq C_{\A_n}(x)$.

Assume to the contrary that $[a] \neq [b]$, that is, $x \neq 0$. Since the norm on $\A_n$ is anisotropic, this means that $n(x) \neq 0$. By Lemma~\ref{lemma:moreno-tilde}(5), $ax = 0$ implies $a\til{x} = 0$, so $\til{x} \in O_{\A_n}(a)$. Hence $\til{x} \in C_{\A_n}(a) \subseteq C_{\A_n}(x)$. Moreover, by Lemma~\ref{lemma:moreno-doubly-pure}(4), the element $x$ is doubly pure. It follows from Lemma~\ref{lemma:quaternionic-subalgebra} that $\til{x}x = -x\til{x} = n(x)\til{e}_0 \neq 0$. We obtain a contradiction with the equality $\til{x}x = x\til{x}$.
\end{proof}

\begin{corollary} \label{corollary:C-equivalent}
Let $\A_n$ be a Cayley--Dickson algebra with anisotropic norm, $a, b \in \A_n$, $\Im(a) \neq 0$, $\Im(b) \neq 0$. If $a$ and $b$ are $C$-equivalent, then $[\Im(a)] = [\Im(b)]$.
\end{corollary}

\begin{proof}
Follows immediately from Theorem~\ref{theorem:C-equivalent}, since $C_{\A_n}(a) = C_{\A_n}(\Im(a))$ and $C_{\A_n}(b) = C_{\A_n}(\Im(b))$.
\end{proof}

\begin{remark}
Let $\A_n$ be a Cayley--Dickson algebra with anisotropic norm, $a \in Z(\A_n)$. Then, by Lemmas~\ref{lemma:tilde-properties}(3) and~\ref{lemma:moreno-tilde}(4--5), the elements $a$ and $\til{a}$ are doubly pure and linearly independent but $O$-equivalent. Besides, if $b \in Z(\A_{n-1})$, then $O_{\A_n}((b,0)) = O_{\A_n}((0,b)) = \{ (c,d) \; | \; c, d \in O_{\A_{n-1}}(b) \}$, hence $(b,0)$ and $(0,b)$ are also doubly pure and linearly independent but $O$-equivalent.
\end{remark}

Lemma~\ref{lemma:A_n-alternative-properties}(1--5) has been proved by Moreno in~\cite[pp. 25--27]{moreno} for real algebras of the main sequence. In his proof Moreno assumes that $c$ and $d$ are alternative elements in $\Main_n$, however, it is valid verbatim for the case when the elements $c,d \in \Main_n$ alternate with $a,b \in \Main_n$ only. Moreno has also proved Lemma~\ref{lemma:alternative-orthogonal-system}(6), but his proof uses that fact that the elements $c$ and $d$ alternate strongly.

Recall that the anti-associator of the elements $a,b,c$ in $\A$ is $\{ a,b,c \} = (ab)c + a(bc)$.

\begin{lemma} \label{lemma:A_n-alternative-properties} \label{lemma:alternative-orthogonal-system}
Let $\A_{n+1}$ be a Cayley--Dickson algebra with anisotropic norm, the elements $c,d \in \A_n$ alternate with $a,b \in \A_n$, $(a,b),(c,d) \in Z(\A_{n+1})$, $(a,b)(c,d) = 0$. Then
\begin{enumerate}[\rm (1)]
    \item $t(a) = t(b) = t(c) = t(d) = 0$;
    \item $n(a) = -\gamma_n n(b)$ and $n(c) = -\gamma_n n(d)$, i.e., $\chi((a,b)) = \chi((c,d)) = -1$;
    \item $[c,a,d] = 2n(c)b$, $[c,b,d] = -2n(d)a$;
    \item $\{ c,a,d \} = \{ c,b,d \} = 0$;
    \item $a \perp b$;
    \item $a, b \in \spn(e_0,c,d,cd)^{\perp}$;
    \item $(c,d)(ac,ad) = 0$.
\end{enumerate}
\end{lemma}

\begin{proof}
\leavevmode
\begin{enumerate}[\rm (1)]
    \item Follows immediately from Lemma~\ref{lemma:moreno-doubly-pure}(4).
    \item Since $(a,b) \neq 0$ and $(c,d) \neq 0$, and the norm on $\A_n$ is anisotropic, we have $n(a) \neq 0$ or $n(b) \neq 0$, and also $n(c) \neq 0$ or $n(d) \neq 0$. By Lemma~\ref{lemma:condition-asterisk}, $\chi = \chi((a,b)) = \chi((c,d)) = \pm 1$. However, if $\chi = 1$, then $n((a,b)) = n((c,d)) = 0$, a contradiction with the fact that the norm on $\A_{n+1}$ is anisotropic. Therefore, $\chi = -1$.
    \item We have $(a,b)(c,d) = (ac + \gamma_n \bar{d}b, da + b\bar{c}) = (ac - \gamma_n db, da - bc) = 0$, hence $ac = \gamma_n db$ and $da = bc$. Thus
    \begin{align*}
        n(d)a &= \bar{d}(da) = -d(bc),\\
        n(d)a &= -\tfrac{1}{\gamma_n}n(c)a = -\tfrac{1}{\gamma_n}(ac)\bar{c} = \tfrac{1}{\gamma_n}(\gamma_n db)c = (db)c,\\
        n(c)b &= (bc)\bar{c} = -(da)c,\\
        n(c)b &= -\gamma_n n(d)b = -\bar{d}(\gamma_n db) = d(ac).
    \end{align*}
    Applying the involution to both sides of each equality, we obtain $n(d)a = -(cb)d = c(bd)$ and $n(c)b = -c(ad) = (ca)d$. Hence $[c,a,d] = 2n(c)b$, $[c,b,d] = -2n(d)a$ and $\{ c,a,d \} = \{ c,b,d \} = 0$, which proves also item~(4).
\setcounter{enumi}{4}
    \item Similarly to the proof of Lemma~\ref{lemma:zero-self-associator}, we consider a skew-symmetric linear operator $S: \A_n \to \A_n$ given by the formula $S(x) = [c,x,d]$ for all $x \in \A_n$. Then $a \perp S(a) = [c,a,d] = 2n(c)b$, so $a \perp b$.
    \item By item~(1), $e_0 \in \spn(a,b,c,d)^{\perp}$. We apply Lemma~\ref{lemma:A_n-next-pair}(1) to the elements $(a,b)$ and $(c,d)$ with $\chi = \chi((c,d)) = -1$ and use the equality $\overline{ac} = \bar{c}\bar{a} = ca$. Then $(c,d)(ca,da)=0$ and, by Lemma~\ref{lemma:moreno-doubly-pure}(4), $t(ca) = t(da) = 0$. It follows from Proposition~\ref{proposition:lambda-form} that $\langle a, c \rangle = - \langle a, \bar{c} \rangle = -\frac{1}{2}t(ca) = 0$ and, similarly, $\langle a, d \rangle = 0$, that is, $a \perp c$ and $a \perp d$. By using the equalities $ca = \gamma_n bd$ and $da = bc$, we obtain similarly that $b \perp c$ and $b \perp d$.
    
    It follows from Lemma~\ref{lemma:A_n-anticomm} that the elements $a,b$ anticommute with $c,d$. Then $ac = -ca$ and $ad = -da$, which proves immediately item~(7). It remains to note that, by Lemma~\ref{lemma:inner-product-movement}(2), 
    \begin{align*}
    \langle a, cd \rangle &= \langle a\bar{d}, c \rangle = \langle da, c \rangle = \langle bc, c \rangle = \langle b, c\bar{c} \rangle = n(c) \langle b, e_0 \rangle = 0,\\
    \langle b, cd \rangle &= \langle \bar{c}b, d \rangle = \langle bc, d \rangle = \langle da, d \rangle = \langle a, \bar{d}d \rangle = n(d) \langle a, e_0 \rangle = 0, 
    \end{align*}
    hence $a \perp cd$ and $b \perp cd$. \qedhere
\end{enumerate}
\end{proof}

We now generalize some results which were obtained in Subsection~4.2 of the author's paper~\cite{our_orthographs1} for the case of real algebras of the main sequence. We remark that they contained an inaccuracy, namely, the existed proof of pairwise orthogonality of the elements $a, b, c, d$ with respect to $\langle \cdot, \cdot \rangle$ used the fact that the elements $a, b, c, d$ alternate strongly pairwise, but all statements were formulated in the assumption that the elements $a, b$ alternate strongly with $c, d$ only (see~\cite[Corollary~4.4, Lemma~4.6]{our_orthographs1}).

In the statements~\ref{corollary:double-hexagon}--\ref{corollary:linear-independence} and in Figure~\ref{figure:double-hexagon} we assume that $\A_{n+1}$ is a Cayley--Dickson algebra with anisotropic norm, $(a,b),(c,d) \in Z(\A_{n+1})$, $(a,b)(c,d) = 0$, and the elements $a,b \in \A_n$ alternate strongly with $c,d \in \A_n$, i.e., $[x,x,y] = [y,y,x] = 0$ for $x \in \{ a, b \}$ and $y \in \{ c, d\}$.

\begin{corollary} \label{corollary:double-hexagon}
There exists the following $6$-cycle in $\Gamma_O(\A_{n+1})$:
$$
(a,b) \leftrightarrow (c,d) \leftrightarrow (ac,ad) \leftrightarrow (a,-b) \leftrightarrow (c,-d) \leftrightarrow (ac,-ad) \leftrightarrow (a,b).
$$
\end{corollary}

\begin{proof}
We use Corollary~\ref{corollary:A_n-double-hexagon} for $\chi = \chi((c,d)) = -1$. By Lemma~\ref{lemma:alternative-orthogonal-system}(6), we have $\overline{ac} = -ac$ and $da = -ad$. Finally, Corollary~\ref{corollary:zero-divisors-symmetry} implies that directed edges of the hexagon in $\Gamma_Z(\A_{n+1})$ correspond to undirected ones in $\Gamma_O(\A_{n+1})$.
\end{proof}

\begin{descript} \label{description:double-hexagon}
By using Lemma~\ref{lemma:moreno-tilde}(5) and Corollary~\ref{corollary:double-hexagon} we obtain a subgraph of $\Gamma_O(\A_{n+1})$ which we call a {\em double hexagon}. It is depicted in Figure~\ref{figure:double-hexagon}. A double hexagon consists of $6$ bipartite graphs $K_{2,2}$ glued together. Note that it contains all hexagons from Figure~\ref{figure:directed-hexagon}.
\end{descript}

\begin{figure}[ht]
\centering
\includegraphics[width=0.74\linewidth]{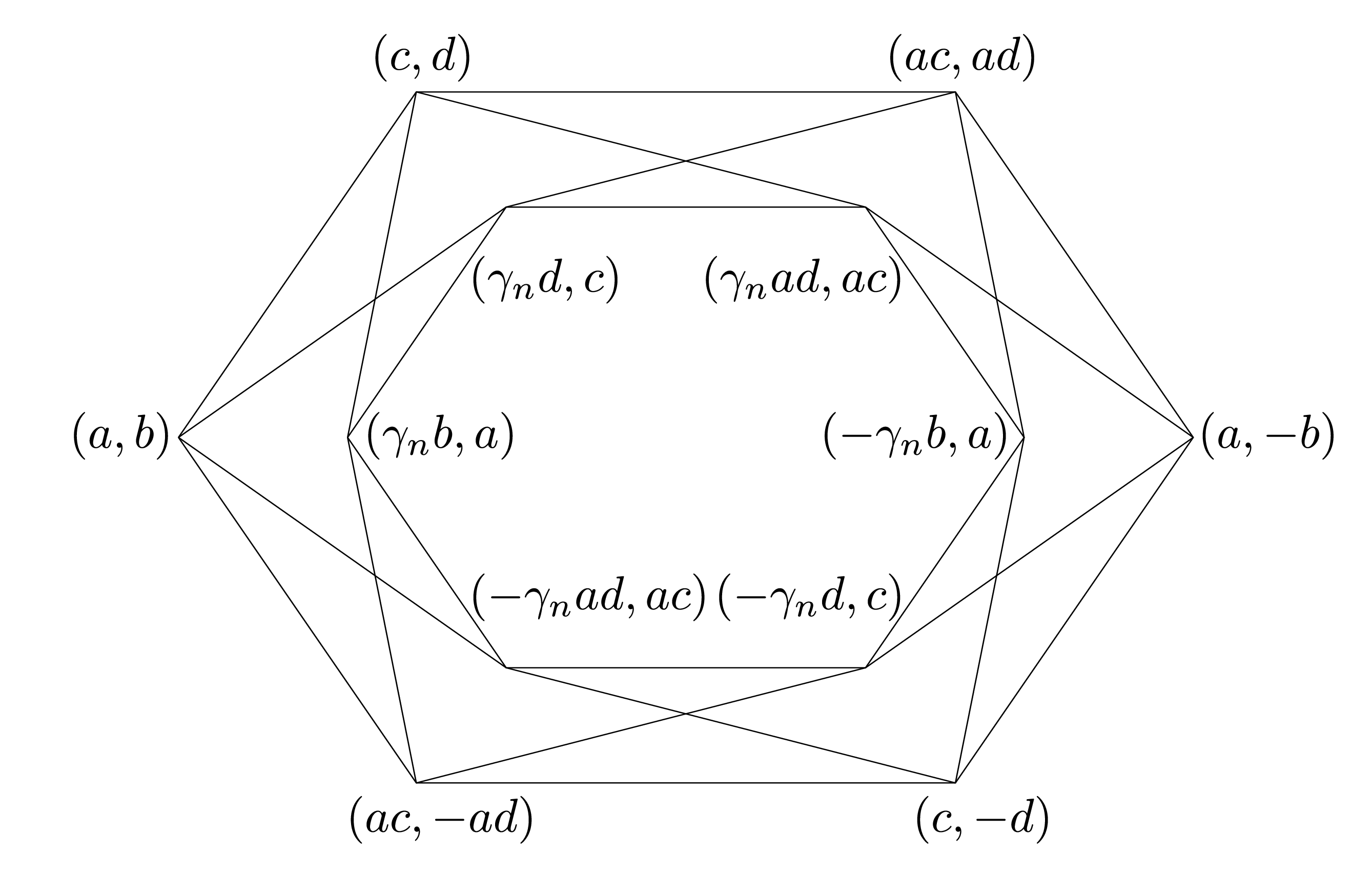}
\caption{\label{figure:double-hexagon} A double hexagon.}
\end{figure}

\begin{lemma} \label{lemma:orthonormal-system}
\leavevmode
\begin{enumerate}[\rm (1)]
\item The elements $e_0,a,b,c,d$ are orthogonal with respect to $\langle \cdot, \cdot \rangle$.
\item The elements $e_0,a,b,c,d,ac,ad$ are orthogonal with respect to $\langle \cdot, \cdot \rangle$.
\end{enumerate}
\end{lemma}

\begin{proof}
\leavevmode
\begin{enumerate}[\rm (1)]
\item Follows immediately from Lemma~\ref{lemma:A_n-alternative-properties}(5--6), applied to pairs of elements $(a,b)(c,d)=0$ and $(c,d)(a,b)=0$.
\item By Lemma~\ref{lemma:strongly-alternative-system}, the elements $ac, ad$ alternate strongly with $a,b,c,d$. By Corollary~\ref{corollary:double-hexagon}, $(a,b)(c,d) = (c,d)(ac,ad) = (ac,ad)(a,-b) = 0$. It remains to use item~(1) thrice. \qedhere
\end{enumerate}
\end{proof}

\begin{corollary} \label{corollary:linear-independence}
All elements in the vertices of the double hexagon in Figure~\ref{figure:double-hexagon} are linearly independent.
\end{corollary}

\begin{proof}
Follows immediately from Lemma~\ref{lemma:orthonormal-system}(2), since, by Lemma~\ref{lemma:alternative-elements-are-normed}, $n(ac) = n(a)n(c) \neq 0$ and $n(ad) = n(a)n(d) \neq 0$.
\end{proof}

Combining the results of Lemmas~\ref{lemma:strongly-alternative-system} and~\ref{lemma:orthonormal-system}(2), Corollaries~\ref{corollary:double-hexagon} and~\ref{corollary:linear-independence}, and Description~\ref{description:double-hexagon}, we obtain the following theorem.

\begin{theorem}
\label{theorem:double-hexagon}
Let $\A_{n+1}$ be a Cayley--Dickson algebra with anisotropic norm, the elements $a,b \in \A_n$ alternate strongly with $c,d \in \A_n$, $(a,b),(c,d) \in Z(\A_{n+1})$, $(a,b)(c,d) = 0$. Then
\begin{enumerate}[\rm (1)]
    \item The elements $ac,ad$ alternate strongly with $a,b,c,d$. \label{item:strong-alternativity}
    \item The elements $e_0,a,b,c,d,ac,ad$ are orthogonal with respect to $\langle \cdot, \cdot \rangle$. \label{item:orthonormal-system}
    \item There exists a subgraph of $\Gamma_O(\A_{n+1})$ which is depicted in Figure~\ref{figure:double-hexagon} and called a double hexagon. \label{item:double-hexagon}
    \item All elements in the vertices of the double hexagon are linearly independent. \label{item:linear-independence}
\end{enumerate}
\end{theorem}

In~\cite[Theorem~4.11]{our_orthographs1} the multiplication table of the vertices of the double hexagon was obtained in the case when $\A_{n+1}$ is a real algebra of the main sequence $\Main_{n+1}$. This result can also be generalized to an arbitrary Cayley--Dickson algebra with anisotropic norm, but the new multiplication table will depend on the parameters $n(a)$, $n(c)$ and~$\gamma_n$. If $\A_{n+1} = \Main_{n+1}$, then $\gamma_n = -1$, and we may assume without loss of generality that $n(a) = n(c) = 1$, so all coefficients in the multiplication table of the vertices of the double hexagon~\cite[p.~677, Table~1]{our_orthographs1} are constant.

\section{Dimensions of annihilators} \label{section:annihilator-dimensions}

In this section we extend the proof of Theorem~9.8 from the paper~\cite{biss} by Biss, Dugger, and Isaksen, which states that the dimension of annihilator of any element in a real algebra of the main sequence is divisible by four, to the case of an arbitrary Cayley--Dickson algebra $\A_n$ with anisotropic norm over a field $\mathbb{F}$, $\chrs \mathbb{F} \neq 2$. But first we show that this statement may fail for Cayley--Dickson algebras with isotropic norm. Recall that, by Lemma~\ref{lemma:two-sided-kernel}, for any $a \in \A_n$ we have $\dim(l.\Ann_\A(a)) = \dim(r.\Ann_\A(a))$.

\begin{lemma}
Let $n \geq 1$, $a \in \A_n$, $t(a) = 0$. Then $\dim(r.\Ann_{\A_n}(a))$ is even.
\end{lemma}

\begin{proof}
Since $t(a) = 0$, by Lemma~\ref{lemma:inner-product-movement}(2) we have $\langle L_a(b), c \rangle = \langle ab, c \rangle = \langle b, \bar{a}c \rangle = - \langle b, L_a(c) \rangle$ for all $b, c \in \A_n$, that is, $L_a$ is a skew-symmetric linear operator with respect to the nondegenerate symmetric bilinear form $\langle \cdot, \cdot \rangle$. Since $\chrs \mathbb{F} \neq 2$, it follows that the rank of $L_a$ is even. But $\dim \A_n = 2^n$ is even, so $\dim(r.\Ann_{\A_n}(a)) = \dim(\Ker L_a)$ is also even.
\end{proof}

\begin{proposition} {\rm \cite[Lemma~4.18]{our_split-algebras}, \cite[Corollary~4.5]{our_split-sedenions}}
Let $\A_n$ be a real low-dimensional Cayley--Dickson split-algebra, i.e., $\mathbb{F} = \mathbb{R}$, $\A_n = \Hyp_n$ and $1 \leq n \leq 4$, and let also $a \in \A_n$.
\begin{enumerate}[(1)]
    \item If $1 \leq n \leq 3$, then $\dim(r.\Ann_{\A_n}(a)) \in \{ 0, 2^{n-1}, 2^n \}$.
    \item If $n = 4$, then $\dim(r.\Ann_{\A_n}(a)) \in \{ 0, 4, 8, 16 \}$.
\end{enumerate}
Hence, for $n \in \{ 3, 4 \}$, $\dim(r.\Ann_{\A_n}(a))$ is divisible by four.
\end{proposition}

The next two examples show that for $n \geq 4$ the algebra $\A_n$ may contain pure doubly alternative elements such that the dimension of their annihilators is even but not divisible by four.

\begin{example} {\rm \cite[Example~4.17]{our_split-algebras}} \label{example:non-split-algebra-zero-divisors}
Let $n \geq 4$, $\mathbb{F} = \mathbb{R}$, $\mathcal{A}_n = \Hyp_{n-1} \{ 1 \}$. Consider $a = e_1^{(n-2)} \in \Main_{n-2}$. Then, by~\cite[Lemma~4.16]{our_split-algebras}, the element $A = (2a + \til{a}, a) \in \mathcal{A}_n$ is pure and doubly alternative. However, $n(a) = 1$ and $n(2a + \til{a}) = 3$, so $\chi(A) = \gamma_{n-1} \frac{n(a)}{n(2a + \til{a})} = \frac{1}{3} \neq \pm 1$, and thus $A$ does not satisfy condition~\eqref{equation:norm-condition}. Moreover, one can obtain from Theorem~\ref{theorem:octonionic-subalgebra} and Lemma~\ref{lemma:double-alternative-annihilators} that
$$
r.\Ann_{\A_n}(A) = \left\{ (c, -c) \; \Big | \; c = (x,-x), \; x \in \spn(e_0, a)^{\perp} \subseteq \Main_{n-2} \right\}.
$$
Hence $\dim(r.\Ann_{\A_n}(A)) = \dim(\spn(e_0, a)^{\perp}) = 2^{n-2} - 2$ is even but not divisible by four.
\end{example}

\begin{example} \label{example:annihilator-dimension-2}
Let $\mathbb{F} = \mathbb{R}$, $\mathcal{A}_4 = \Hyp_3 \{ 1 \}$. Consider $a = e_2 + e_5 + e_6$, $b = e_0 + e_1 + e_5 \in \Hyp_3$. Since $\Hyp_3$ is an alternative algebra, the element $(a, b) \in \A_4$ is pure and doubly alternative. Moreover, $n(b) = -n(a) = 1$, so $\chi((a,b)) = \gamma_{n-1} \frac{n(b)}{n(a)} = -1$, that is, $(a,b)$ satisfies condition~\eqref{equation:norm-condition}. However,
\begin{align*}
r.\Ann_{\A_4}((a,b)) = \spn\big(&(-e_2 + e_3 - e_6 - e_7, e_2 + e_3 + e_6 - e_7),\\
&(e_1 + e_2 + 2e_3 - 2e_4 + e_5 - e_7, e_1 - 3e_3 + 2e_4 + e_5 - e_6 + 2e_7)\big).
\end{align*}
Hence $\dim(r.\Ann_{\A_4}((a,b))) = 2$ is even but not divisible by four.
\end{example}

The following example shows that if an element of $\A_n$ is not pure, then for a sufficiently large $n$ the dimension of its annihilator may be odd.

\begin{example} \label{example:annihilator-dimension-3}
Let $\mathbb{F} = \mathbb{R}$, $\mathcal{A}_5 = \Hyp_5 = \Main_4 \{ -1 \}$. Consider 
\begin{align*}
    a &= e_0 - e_1 + 2 e_2 + 2 e_5 - e_6 + 2 e_7 + e_9 - e_{10} - e_{11} - 2e_{12} - 2e_{13} - 2e_{14} - 2e_{15},\\
    b &= -e_0 - 2e_1 + e_2 - 2e_3 - 2e_4 - 2e_6 + 2e_7 + e_9 - 2e_{10} + 2e_{11} + e_{12} + e_{13} - e_{15} \in \Main_4.
\end{align*}
One can verify that in this case $\dim(r.\Ann_{\Hyp_5}((a,b))) = 3$ is odd.
\end{example}

\bigskip

Let now $\A_n$ be a Cayley--Dickson algebra with anisotropic norm over a field~$\mathbb{F}$, $\chrs \mathbb{F} \neq 2$. Recall that, by Lemma~\ref{lemma:moreno-zero-divisor-conditions}(2), in this case we have $l.\Ann_{\A_n}(a) = r.\Ann_{\A_n}(a) = O_{\A_n}(a)$ for all $a \in \A_n$.

\begin{lemma} \label{lemma:new-field}
Let $n \geq 1$. We denote $\mathbb{K} = \spn(e_0, \til{e}_0)$. Then
\begin{enumerate}[{\rm (1)}]
    \item $\mathbb{K}$ is a field;
    \item $\A_n$ is a left vector space over $\mathbb{K}$.
\end{enumerate}
\end{lemma}

\begin{proof}
\leavevmode
\begin{enumerate}[{\rm (1)}]
    \item Since $(\til{e}_0)^2 = \gamma_{n-1}$, the set $\mathbb{K}$ is closed under addition and multiplication, that is, it forms a subalgebra in $\A_n$. Moreover, multiplication on $\mathbb{K}$ is commutative and associative. Since the norm on $\A_n$ is anisotropic, its restriction to $\mathbb{K}$ is also anisotropic, and thus $\mathbb{K}$ does not contain zero divisors. Therefore, $\mathbb{K}$ is a field.
    \item It is sufficient to show that $[k_1, k_2, a] = 0$ for all $k_1, k_2 \in \mathbb{K}$ and $a \in \A_n$. This is equivalent to the equality $[\til{e}_0, \til{e}_0, a] = 0$ for all $a \in \A_n$, that is, to alternativity of the element $\til{e}_0$. But $\til{e}_0 = (0, e_0)$ is one of the standard basis elements, and thus it is alternative by~\cite[Lemma~4]{schafer}. \qedhere
\end{enumerate}
\end{proof}

By Lemma~\ref{lemma:inner-product-movement}(2), $\langle a, b \rangle = \langle a\bar{b}, e_0 \rangle$, that is, $\langle a, b \rangle$ is an orthogonal projection of $a\bar{b}$ onto $\mathbb{F}$ with respect to $\langle \cdot, \cdot \rangle$. We can use this observation to define a $\mathbb{K}$-valued Hermitian inner product on $\A_n$.

\begin{notation}
Given $a, b \in \A_n$, we denote by $\langle a, b \rangle_{\mathbb{K}}$ the orthogonal projection of $a\bar{b}$ onto $\mathbb{K}$ with respect to $\langle \cdot, \cdot \rangle$.
\end{notation}

\begin{lemma} \label{lemma:new-inner-product}
\leavevmode
\begin{enumerate}[{\rm (1)}]
    \item For any $a, b \in \A_n$ it holds that $\langle a, b \rangle_{\mathbb{K}} = \langle a, b \rangle + \frac{1}{\gamma_{n-1}} \langle \til{e}_0 a, b \rangle \til{e}_0$.
    \item $\langle a, b \rangle_{\mathbb{K}}$ is a Hermitian inner product, that is, an anisotropic (in particular, nondegenerate) Hermitian sesquilinear form.
\end{enumerate}
\end{lemma}

\begin{proof}
\leavevmode
\begin{enumerate}[{\rm (1)}]
    \item By Lemma~\ref{lemma:inner-product-movement}(2), $\langle a\bar{b}, e_0 \rangle = \langle a, b \rangle$ and $\langle a\bar{b}, \til{e}_0 \rangle = \langle a, \til{e}_0 b \rangle = -\langle \til{e}_0 a, b \rangle$. Since $e_0 \perp \til{e}_0$, $n(e_0) = 1$ and $n(\til{e}_0) = -\gamma_{n-1}$, we have
    $$
    \langle a, b \rangle_{\mathbb{K}} = \frac{\langle a\bar{b}, e_0 \rangle}{n(e_0)} e_0 + \frac{\langle a\bar{b}, \til{e}_0 \rangle}{n(\til{e}_0)} \til{e}_0 = \langle a, b \rangle + \frac{\langle \til{e}_0 a, b \rangle}{\gamma_{n-1}} \til{e}_0.
    $$
    \item Additivity in both arguments is obvious. We now show that $\langle ka, b \rangle_{\mathbb{K}} = k \langle a, b \rangle_{\mathbb{K}}$ for all $k \in \mathbb{K}$ and $a, b \in \A_n$. By $\mathbb{F}$-linearity, it is sufficient to prove this statement for $k = e_0$ and $k = \til{e}_0$ only. The first case is evident, while in the second case we obtain
    \begin{align*}
        \langle \til{e}_0 a, b \rangle_{\mathbb{K}} &= \langle \til{e}_0 a, b \rangle + \frac{\langle \til{e}_0 (\til{e}_0 a), b \rangle}{\gamma_{n-1}} \til{e}_0 = \langle \til{e}_0 a, b \rangle + \frac{\langle (\til{e}_0)^2 a, b \rangle}{\gamma_{n-1}} \til{e}_0 =\\
        &= \langle \til{e}_0 a, b \rangle + \langle a, b \rangle \til{e}_0 = \langle a, b \rangle \til{e}_0 + \frac{\langle \til{e}_0 a, b \rangle}{\gamma_{n-1}} (\til{e}_0)^2 = \til{e}_0 \langle a, b \rangle_{\mathbb{K}}.
    \end{align*}
    Moreover, it follows from $a\bar{b} = \overline{b\bar{a}}$ that $\langle a, b \rangle_{\mathbb{K}} = \overline{\langle b, a \rangle_{\mathbb{K}}}$. Since $a\bar{a} = n(a) \in \mathbb{F} \subseteq \mathbb{K}$, we have $\langle a, a \rangle_{\mathbb{K}} = n(a)$ for $a \in \A_n$, and thus $\langle \cdot, \cdot \rangle_{\mathbb{K}}$ is anisotropic. \qedhere
\end{enumerate}
\end{proof}

\begin{lemma} \label{lemma:conjugate-linear-skew-hermitian}
Let $a \in \A_n$ be doubly pure. Then $L_a$ is a conjugate-linear skew-Hermitian mapping in the sense that
\begin{enumerate}[{\rm (1)}]
    \item $L_a(b + c) = L_a(b) + L_a(c)$,
    \item $L_a(kb) = \bar{k} L_a(b)$,
    \item $\langle L_a(b), c \rangle_{\mathbb{K}} = -\overline{\langle b, L_a(c) \rangle_{\mathbb{K}}}$
\end{enumerate}
for all $k \in \mathbb{K}$, $b, c \in \A_n$.
\end{lemma}

\begin{proof}
\leavevmode
\begin{enumerate}[{\rm (1)}]
    \item Additivity of $L_a$ is obvious.
    \item Due to $\mathbb{F}$-linearity, it is sufficient to consider the cases $k = e_0$ and $k = \til{e}_0$. The first case is evident, while in the second case we need to show that $a(\til{e}_0b) = -\til{e}_0(ab)$. Applying the involution to both sides of the required equality and using the fact that $a$ is pure, we conclude that it is equivalent to $(\bar{b}\til{e}_0)a = -(\bar{b}a)\til{e}_0$. Since $a$ is doubly pure, we obtain by Lemma~\ref{lemma:tilde-properties}(2) that $\til{\bar{b}} a = -\til{\bar{b}a}$, as required.
    \item We use Lemmas~\ref{lemma:inner-product-movement}(2) and~\ref{lemma:new-inner-product}(1) and item~(2) to obtain that
    \begin{align*}
        \langle L_a(b), c \rangle_{\mathbb{K}} &= \langle ab, c \rangle_{\mathbb{K}} = \langle ab, c \rangle + \frac{\langle \til{e}_0 (ab), c \rangle}{\gamma_{n-1}} \til{e}_0 = \langle b, \bar{a}c \rangle + \frac{\langle b, \bar{a}(\overline{\til{e}_0} c) \rangle}{\gamma_{n-1}} \til{e}_0 =\\
        &=-\langle b, ac \rangle + \frac{\langle b, a(\til{e}_0 c) \rangle}{\gamma_{n-1}} \til{e}_0 = -\langle b, ac \rangle + \frac{\langle b, -\til{e}_0 (ac) \rangle}{\gamma_{n-1}} \til{e}_0 =\\
        &= -\langle b, ac \rangle + \frac{\langle \til{e}_0 b, (ac) \rangle}{\gamma_{n-1}} \til{e}_0 = -\overline{\langle b, ac \rangle_{\mathbb{K}}} = -\overline{\langle b, L_a(c) \rangle_{\mathbb{K}}}. \tag*{\qedhere}
    \end{align*}
\end{enumerate}
\end{proof}

In~\cite{biss} the following lemma is formulated for the field of complex numbers only, however, it holds also for the case of an arbitrary field $\mathbb{K}$, $\chrs \mathbb{K} \neq 2$, with an involution $a \mapsto \bar{a}$.

\begin{lemma} {\em \cite[Lemma~6.7]{biss}} \label{lemma:kernel-codimension}
Let $V$ be a finite-dimensional linear space over $\mathbb{K}$ with a nondegenerate Hermitian form, and let also $L$ be a conjugate-linear skew-Hermitian endomorphism of $V$. Then the $\mathbb{K}$-codimension of $\Ker L$ in~$V$ is even.
\end{lemma}

\begin{proof}
We identify $V$ with $\mathbb{K}^m$, where $m = \dim V$. Since $L$ is conjugate-linear, there exists a matrix $A \in M_m(\mathbb{K})$ such that $Lx = \overline{Ax}$ for all $x \in V$. Let $H \in M_m(\mathbb{K})$ be the matrix of the Hermitian form on $V$, that is, $\langle x, y \rangle_{\mathbb{K}} = x^t H \bar{y}$ for all $x, y \in V$. Then we have $\bar{H} = H^t$. The mapping $L$ is skew-Hermitian, hence we obtain
\begin{align*}
    x^t H Ay &= \langle x, \overline{Ay} \rangle_{\mathbb{K}} = \langle x, Ly \rangle_{\mathbb{K}} = -\overline{\langle Lx, y \rangle_{\mathbb{K}}} = -\overline{\langle \overline{Ax}, y \rangle_{\mathbb{K}}} =\\
    &= -\overline{(\overline{Ax})^t H \bar{y}} = -(Ax)^t \bar{H} y = -x^t A^t H^t y = -x^t (HA)^t y 
\end{align*}
for all $x, y \in V$, and thus $HA = -(HA)^t$. Since $\chrs \mathbb{K} \neq 2$, it follows that the rank of $HA$ is even, so nondegeneracy of $H$ implies that the rank of $A$ is also even. Therefore, the $\mathbb{K}$-codimension of $\Ker L = \Ker A$ in~$V$ is even.
\end{proof}

\begin{theorem} \label{theorem:anisotropic-annihilator-dimension}
Let $n \geq 2$, $a \in \A_n$. Then $\dim_{\mathbb{F}}(r.\Ann_{\A_n}(a))$ is divisible by four.
\end{theorem}

\begin{proof}
If $a = 0$, then $\dim_{\mathbb{F}}(r.\Ann_{\A_n}(a)) = \dim_{\mathbb{F}}(\A_n) = 2^n$ is divisible by four, and if $a \neq 0$ and $a \notin Z(\A_n)$, then $\dim_{\mathbb{F}}(r.\Ann_{\A_n}(a)) = 0$ is also divisible by four. We now consider the case when $a \in Z(\A_n)$. By Lemma~\ref{lemma:moreno-doubly-pure}(4), the element $a$ is doubly pure. Hence, by Lemma~\ref{lemma:conjugate-linear-skew-hermitian}, $L_a$ is a conjugate-linear skew-Hermitian mapping. It follows from Lemma~\ref{lemma:kernel-codimension} that $\codim_{\mathbb{K}} \Ker L_a$ is even. Since $\dim_{\mathbb{K}} \A_n = 2^{n-1}$ is even, this means that $\dim_{\mathbb{K}} \Ker L_a$ is also even. Hence $\dim_{\mathbb{F}}(r.\Ann_{\A_n}(a)) = \dim_{\mathbb{F}} \Ker L_a = 2 \dim_{\mathbb{K}} \Ker L_a$ is divisible by four.
\end{proof}

Some other results on the dimension of annihilators, which were obtained in the papers~\cite{biss,biss2} for real algebras of the main sequence, can also be generalized to the case of arbitrary Cayley--Dickson algebras with anisotropic norm. In particular, this is true for Lemma~8.4, Proposition~8.11 and Theorem~13.2 from~\cite{biss}. We remark that, in case of arbitrary Cayley--Dickson algebras with anisotropic norm, Definition~3.1 of the element $\{ a, b \} \in \A_{n+1}$ from~\cite{biss2} should be modified as follows: the factor $\frac{1}{\sqrt{2}}$ should be deleted, and a new factor~$\gamma_n$ should appear in the first component of $\{ a, b \}$.

\medskip

The author is grateful to her scientific advisor Professor Alexander E. Guterman for posing the
problem and fruitful discussions.

\end{document}